\documentclass[a4paper,12pt]{article}

\usepackage{amsmath,amssymb,amsthm}
\usepackage{mathabx}
\usepackage{enumitem}
\usepackage[margin=1.2in]{geometry}
\usepackage{hyperref}
\usepackage{graphicx}
\usepackage{multirow}

\newcommand{\mr}{\mathrm}
\newcommand{\mc}{\mathcal}
\newcommand{\mb}{\mathbf}
\newcommand{\bs}{\boldsymbol}
\newcommand{\mbb}{\mathbb}
\newcommand{\pp}{\partial}
\newcommand{\dualpr}[1]{\langle #1\rangle}

\renewcommand{\div}{\nabla\cdot}

\newtheorem{lemma}{Lemma}[section]
\newtheorem{theorem}{Theorem}[section]
\newtheorem{prop}{Proposition}[section]

\newtheorem{defi}{Definition}[section]
\newtheorem{assmp}{Assumption}[section]

\newcommand{\newparagraph}[1]{\quad\\ \noindent{\bf #1}}

\begin{document}

\title{A unified error analysis of HDG methods for the static Maxwell equations}
\author{Shukai Du\thanks{Email: shukaidu@udel.edu}\,\, and Francisco-Javier Sayas\\
Department of Mathematical Sciences, University of Delaware}
\maketitle

\begin{abstract}
We propose a framework that allows us to analyze different variants of HDG methods for the static Maxwell equations using one simple analysis. It reduces all the work to the construction of projections that best fit the structures of the approximation spaces. 
As applications, we analyze four variants of HDG methods (denoted by $\mathfrak{B}$, $\mathfrak{H}$, $\mathfrak{B+}$, $\mathfrak{H+}$), where two of them are known (variants $\mathfrak{H}$, $\mathfrak{B+}$) and the other two are new (variants $\mathfrak{H+}$, $\mathfrak{B}$). Under certain regularity assumption, we show that all the four variants are optimally convergent and that variants $\mathfrak{B+}$ and $\mathfrak{H+}$ achieve superconvergence without post-processing. For the two known variants, we prove their optimal convergence under weaker requirements of the meshes and the stabilization functions thanks to the new analysis techniques being introduced. For solution with low-regularity, we give an analysis to these methods and investigate the effect of different stabilization functions on the convergence. At the end, we provide numerical experiments to support the analysis.
\end{abstract}


\section{Introduction}
Maxwell equations describe the interaction between electric and magnetic fields and play a central role in modern sciences and engineering. To understand the solution of Maxwell equations in various application scenarios, numerical treatments are necessary. The finite element method (FEM) is one of these numerical tools and it has some nice features such as easy handling of complicate geometry, exponential rate of convergence by hp-refinements, etc.

Finite element methods can be divided into two categories -- conforming and non-conforming. For Maxwell equations, conforming elements usually refer to $H(\mr{curl})$-conforming elements since $H(\mr{curl})$ is used as the energy space for the solution of Maxwell equations. $H(\mr{curl})$-conforming elements (also called edge elements) have been widely studied since they were first proposed by N\'{e}d\'{e}lec in \cite{Ne:1980,Ne:1986}; see, for instance, \cite{GaMe:2012,Hi:2002,Mo:1992,Mo:2003,Mo:1991,ZhShWiXu:2009}.

For non-conforming elements, one popular choice is the discontinuous Galerkin (DG) finite element method (see \cite{ArBrCoMa:2001} for a general introduction and see, for instance, \cite{BrLiSu:2008,CoLiSh:2004,FeWu:2014,HoPeSc:2004,PeScMo:2002,PeSc:2003} for DG methods for Maxwell equations). Since DG methods allow the use of independent approximation spaces on each element, they possess certain nice properties such as the flexibility of choosing local spaces, allowance of triangulation with hanging nodes, high parallel efficiency, easiness of implementation, simple treatment of boundary conditions, etc. Despite their advantages, DG methods in general use more degrees of freedom compared to the corresponding conforming methods. To overcome this difficulty, the hybridizable discontinuous Galerkin (HDG) method was proposed \cite{CoGoLa:2009}. By introducing a Lagrange multiplier on the skeleton of the mesh and using the hybridization techniques, HDG method allows the solution of a much smaller system only involving the Lagrange multiplier and then to recover locally the rest of the degrees of freedom on each element. 

Recently, there has been considerable interest in developing HDG methods for Maxwell equations and many variants \cite{ChCuXu:2019,ChQiSh:2018,ChQiShSo:2017,LiLaPe:2015,LuChQi:2017,NgPeCo2:2011} of HDG methods have been proposed and analyzed. However, to the best of our knowledge, there is no work that provides a unified point of view of understanding these variants. This leads to a possibility of repeated or unnecessary arguments being generated and a lack of recognition of the connections among these variants. This motivates us to consider a unified analysis. In this paper, we propose a framework that enables us to clearly decouple the error analysis techniques into two groups -- those related to the PDE and those related to the HDG variants (namely, the choices of the approximation spaces and stabilization functions). 
The benefits of doing so include the following:
\begin{itemize}
\item Recycling existing error analysis techniques. We demonstrate this by using only one analysis to obtain the error estimates for four variants of HDG methods. In this way, we can avoid introducing repeated arguments for each variants.
\item Providing guidelines for systematically discovering new optimal convergent and super convergent HDG methods. We discover two new HDG variants $\mathfrak{B}$ and $\mathfrak{H+}$ by using this framework, where variant $\mathfrak{H+}$ achieves superconvergence in the sense of the degrees of freedom of the numerical trace (the discrete electric field achieves $\mc O(h^{k+2})$ convergence while its numerical trace only lives in a proper subspace of $\mc P_{k+1}(F)^t$ on each face $F$; see the end of Section \ref{sec:sys_err_ana} for a detailed discussion about this).
\item Simple analysis of mixed type HDG methods where the local spaces and stabilization functions vary from element to element. This is doable since we use local projections to capture the features of the HDG variants (the main part of which is how to choose local approximation spaces and stabilization functions).
\end{itemize}

Let us mention two inspirations of this work. The first one is \cite{CoGoSa:2010}, where a tailored projection is proposed to analyze a class of HDG methods in a unified way under the setting of elliptic problem (this is inspired by the celebrated Raviart-Thomas (RT) and Brezzi-Douglas-Marini (BDM) projections). This approach to the analysis is often referred to as ``projection-based error analysis'' (we refer to \cite{DuSa_book:2019} for a systematic introduction). The second one is our previous work about HDG methods for elastic waves \cite{du2019new}, in which we show that we can use projection-based error analysis for those HDG methods whose approximation spaces do not admit $M$-decomposition \cite{CoFuSa:2017}. The work of this paper can be regarded as a generalization of the work in \cite{du2019new} to the setting of Maxwell equations.

To proceed with the discussion, we shall now introduce the model problem.  Let $\Omega\subset\mbb R^3$ be a bounded simply connected polyhedral domain with connected Lipschitz boundary $\Gamma:=\pp\Omega$. We consider the following static Maxwell equations in a mixed form:
\begin{subequations}\label{eq:PDE}
\begin{alignat}{5}
\label{eq:PDE_1}
\mb w-\nabla\times\mb u &= 0 &\quad& \mr{in}\ \Omega,\\
\label{eq:PDE_2}
\nabla\times\mb w+\nabla p&=\mb  f &&\mr{in}\ \Omega,\\
\label{eq:PDE_3}
\nabla\cdot \mb u &=0 && \mr{in}\ \Omega,\\
\label{eq:PDE_4}
\mb n\times\mb u&=\mb g &&\mr{on}\ \Gamma,\\
\label{eq:PDE_5}
p &= 0 &&\mr{on}\ \Gamma.
\end{alignat}
\end{subequations}
In the above, variables $\mb u$ and $\mb w$ are the electric and the magnetic fields respectively, and $p$ is a Lagrange multiplier introduced to have a better control of $\nabla\cdot\mb u$ (see \cite{BoGu:2011,BoGuLu:2013}). Note that when $\mb f$ is divergence free, $p$ admits trivial solution. We remark that \eqref{eq:PDE} with a different boundary condition is related to the Stokes equations with vorticity formulations; see, for instance, \cite{CoGo:2005,CoGo2:2005}.

The rest of the paper is organized as follows. In Section \ref{sec:gen_setting}, we propose an HDG framework with unspecified approximation spaces and stabilization functions; we then give an analysis by using a projection satisfying certain criteria. In Section \ref{sec:gal_proj}, we review some well known projections and construct some new projections that we shall use later. In Section \ref{sec:sys_err_ana}, we consider four variants of HDG methods for Maxwell equations (denoted by $\mathfrak{B}$, $\mathfrak{H}$, $\mathfrak{B+}$, $\mathfrak{H+}$). We give a unified analysis to the four variants by using the abstract analysis setting established in Section \ref{sec:gen_setting} combined with suitable projections discussed in Section \ref{sec:gal_proj}. We show that all the variants are optimal and variants $\mathfrak{B+}$, $\mathfrak{H+}$ achieve superconvergence under certain regularity assumption. We then compare these four variants and discuss their connections. In  Section \ref{sec:low_reg}, we give an analysis to these methods and also the standard HDG method for solution with low-regularity, and investigate the effect of different stabilization functions on the convergence.
Finally in Section \ref{sec:num_exp}, we present some numerical tests to support the analysis.

\section{The framework}\label{sec:gen_setting}
\subsection{Notation}
We begin by introducing some notation that will be used extensively in the paper.
Let $\mc T_h$ be a conforming triangulation of $\Omega$, where each element $K\in\mc T_h$ is a star-shaped polyhedron. Let $\mc E_K$ and $\mc E_h$ be the collections of all faces of $K$ and $\mc T_h$, respectively. We use the standard notation $h_K$ as the diameter of $K$ and denote by $h:=\max_{K\in\mc T_h}h_K$ as the mesh size of $\mc T_h$.
For $k\ge0$, we denote by $\mc P_k(\mc O)$ the polynomial space of degree $k$ supported on $\mc O$, where $\mc O$ can be an element in $\mc T_h$ or a face in $\mc E_h$. Let $N_0$ be a large integer.
For any $K\in\mc T_h$, let $W(K)$ and $V(K)$ be two subspaces of $\mc P_{N_0}(K)^3$, and $Q(K)$ be a subspace of $\mc P_{N_0}(K)$. For any $F\in\mc E_h$, let $N(F)$ be a subspace of $\mc P_{N_0}(F)^t:=\{\mb u\in\mc P_{N_0}(F)^3:\mb u\cdot\mb n_F=0\}$ (for a vector field $\mb v$ supported on certain surface $F$, we denote by $\mb v^t:=\mb n_F\times\mb v\times\mb n_F$ the tangential component of the vector field), and $M(F)$ be a subspace of $\mc P_{N_0}(F)$. Denote by $N(\pp K):=\prod_{F\in\mc E_K} N(F)$ and $M(\pp K):=\prod_{F\in\mc E_K} M(F)$.  Let $\mr P_N: L^2(\pp K)^3\rightarrow N(\pp K)$ and $\mr P_M: L^2(\pp K)\rightarrow M(\pp K)$ be the $L_2$ projections to their range spaces respectively. We assume all the spaces introduced above are non-empty. Define
\begin{subequations}
\begin{alignat*}{5}
&W_h:=\prod_{K\in\mc T_h}W(K),\quad V_h:=\prod_{K\in\mc T_h}V(K),\quad Q_h:=\prod_{K\in\mc T_h}Q(K),\\
&N_h:=\prod_{F\in\mc E_h}N(F),\quad M_h:=\prod_{F\in\mc E_h}M(F).
\end{alignat*}
\end{subequations}
We use the following notation for the discrete inner products on $\mc T_h$ and $\pp\mc T_h$:
\begin{align*}
(*_1,*_2)_{\mc T_h}=\sum_{K\in\mc T_h}(*_1,*_2)_K,\quad
\dualpr{*_1,*_2}_{\pp\mc T_h}=\sum_{K\in\mc T_h}\dualpr{*_1,*_2}_{\pp K},
\end{align*}
where $(\cdot,\cdot)_K$ and $\dualpr{\cdot,\cdot}_{\pp K}$ denote the $L_2$ inner products on $K$ and $\pp K$ respectively.

\subsection{HDG methods}
Depending on the choices of the approximation spaces $\{W(K),V(K),Q(K)\}_{K\in\mc T_h}$ and $\{N(F),M(F)\}_{F\in\mc E_h}$, we obtain different variants of HDG methods. 
We assume these spaces satisfy the following conditions:
\begin{subequations}\label{eq:hdg_sp}
\begin{alignat}{5}
\label{eq:hdg_sp_1}
\nabla\times V(K)\subset W(K),\\
\label{eq:hdg_sp_2}
\nabla\cdot V(K)\subset Q(K),\\
\label{eq:hdg_sp_3}
\nabla\times W(K)+\nabla Q(K)\subset V(K),\\
\label{eq:hdg_sp_4}
\mb n_{\pp K}\times W(K)\subset N(\pp K),\\
\label{eq:hdg_sp_5}
\gamma_{\pp K}Q(K)+V(K)\cdot\mb n_{\pp K}\subset M(\pp K).
\end{alignat}
\end{subequations}
All the HDG variants we will study in this paper satisfy \eqref{eq:hdg_sp} and we assume these conditions hold throughout the paper. 
We now give the HDG scheme under this general setting:\\
Find $(\mb w_h,\mb u_h,p_h,\widehat{\mb u}_h,\widehat{p}_h)
\in W_h\times V_h\times Q_h\times N_h\times M_h$ such that
\begin{subequations}\label{eq:HDG_scheme}
\begin{alignat}{5}
\label{eq:HDG_scheme_1}
(\mb w_h,\mb r)_{\mc T_h}-(\mb u_h,\nabla\times\mb r)_{\mc T_h}-\dualpr{\widehat{\mb u}_h,\mb r\times\mb n}_{\pp\mc T_h}&=0,\\
\label{eq:HDG_scheme_2}
(\nabla\times\mb w_h,\mb v)_{\mc T_h}+\dualpr{\tau_t\mr P_N(\mb u_h-\widehat{\mb u}_h),\mb v}_{\pp\mc T_h}&\\
\nonumber
-(p_h,\nabla\cdot\mb v)_{\mc T_h}+\dualpr{\widehat{p}_h,\mb v\cdot\mb n}_{\pp\mc T_h}
&=(\mb f,\mb v)_{\mc T_h},\\
\label{eq:HDG_scheme_3}
(\nabla\cdot\mb u_h,q)_{\mc T_h}+\dualpr{\tau_n(p_h-\widehat{p}_h),q}_{\pp\mc T_h}
&=0,\\
\label{eq:HDG_scheme_4}
-\dualpr{\mb n\times\mb w_h+\tau_t(\mb u_h-\widehat{\mb u}_h),\bs\eta}_{\pp\mc T_h\backslash\Gamma}&=0,\\
\label{eq:HDG_scheme_5}
-\dualpr{\widehat{\mb u}_h,\bs\eta}_\Gamma &= -\dualpr{\mb g\times\mb n,\bs\eta}_\Gamma,\\
\label{eq:HDG_scheme_6}
-\dualpr{\mb u_h\cdot\mb n+\tau_n(p_h-\widehat{p}_h),\mu}_{\pp\mc T_h\backslash\Gamma}
&=0,\\
\label{eq:HDG_scheme_7}
-\dualpr{\widehat{p}_h,\mu}_\Gamma&=0,
\end{alignat}
\end{subequations}
for all $(\mb r,\mb v,q,\bs\eta,\mu)
\in W_h\times V_h\times Q_h\times N_h\times M_h$. In the above equations \eqref{eq:HDG_scheme}, the two stabilization functions $\tau_t,\tau_n\in\prod_{K\in\mc T_h}\prod_{F\in\mc E_K}\mc P_0(F)$ and we assume $\tau_t\big|_{\pp K},\tau_n\big|_{\pp K}\ge0$ for all $K\in\mc T_h$. 

\begin{prop}\label{prop:unique_solv}
If $\tau_t,\tau_n>0$ and
\begin{align}\label{eq:prf_2}
\{\mb n\times\mb u\times\mb n:\ \mb u\in V(K),\ \nabla\times\mb u=0\}\subset N(\pp K)
\end{align}
for all $K\in\mc T_h$, then \eqref{eq:HDG_scheme} is uniquely solvable.
\end{prop}
\begin{proof}
It is obvious that \eqref{eq:HDG_scheme} is a square system. Let $\mb f=0$ and $\mb g=0$, we aim to show that the system only admits trivial solution. By taking $\mb r=\mb w_h$, $\mb v=\mb u_h$, $q=p_h$ in equations \eqref{eq:HDG_scheme_1}-\eqref{eq:HDG_scheme_3}, $\bs\eta=\widehat{\mb u}_h$ in \eqref{eq:HDG_scheme_4} and $\bs\eta=\mb n\times\mb w_h+\tau_t(\mb P_N\mb u_h-\widehat{\mb u}_h)$ in \eqref{eq:HDG_scheme_5}, $\mu=\widehat{p}_h$ in \eqref{eq:HDG_scheme_6} and $\mu=\mb u_h\cdot\mb n+\tau_n(p_h-\widehat{p}_h)$ in \eqref{eq:HDG_scheme_7}, then adding up all these equations, we obtain
\begin{equation}
(\mb w_h,\mb w_h)_{\mc T_h}+\dualpr{\tau_t(\mb P_N\mb u_h-\widehat{\mb u}_h),(\mb P_N\mb u_h-\widehat{\mb u}_h)}_{\pp\mc T_h}
+\dualpr{\tau_n(p_h-\widehat{p}_h),(p_h-\widehat{p}_h)}_{\pp\mc T_h}=0.
\end{equation}
Since $\tau_t,\tau_n>0$, we have
\begin{align}\label{eq:prf_1}
\mb w_h=0,\quad \mb P_N\mb u_h-\widehat{\mb u}_h=0,\quad p_h-\widehat{p}_h=0.
\end{align}
Equation \eqref{eq:prf_1} with \eqref{eq:HDG_scheme_2} implies $(\nabla p_h,\mb v)_{\mc T_h}=0$ for all $\mb v\in V_h$. This with \eqref{eq:hdg_sp_3} implies $p_h\equiv c_K$ on any $K\in\mc T_h$. Now we use the fact that $p_h-\widehat{p}_h=0$ and obtain $p_h\equiv c$ on $\Omega$, which with \eqref{eq:HDG_scheme_7} implies $p_h\equiv0$.

By \eqref{eq:prf_1}, \eqref{eq:HDG_scheme_1} and \eqref{eq:hdg_sp_4}, we have
\begin{align*}
0=\dualpr{\mb P_N\mb u_h-\widehat{\mb u}_h,\mb r\times\mb n}_{\pp\mc T_h}=
(\nabla\times\mb u_h,\mb r)_{\mc T_h}\quad\forall\mb r\in W_h.
\end{align*}
The above equation with \eqref{eq:hdg_sp_1} implies $(\nabla\times)\big|_K\mb u_h=0$. Now, by \eqref{eq:prf_2}, we have $\widehat{\mb u}_h=\mb P_N\mb u_h=\mb u_h^t$. This with $(\nabla\times)\big|_K\mb u_h=0$ implies $\nabla\times\mb u_h=0$ on $\Omega$. On the other hand, by \eqref{eq:HDG_scheme_3}, \eqref{eq:HDG_scheme_6}, \eqref{eq:hdg_sp_5}
 and the fact that $p_h-\widehat{p}_h=0$ (see \eqref{eq:prf_1}), we have $\div\mb u_h=0$ on $\Omega$. In conclusion, we have obtained
\begin{align*}
\nabla\times\mb u_h=0,\ \div\mb u_h=0,\ \mb n_{\pp\Omega}\times\mb u_h=0.
\end{align*}
This proves that $\mb u_h=0$ since we have assumed that $\Omega$ is a simply connected domain with connected Lipschitz boundary.
\end{proof}

Note that Proposition \ref{prop:unique_solv} gives a sufficient condition for the unique solvability of a general class of HDG methods. However, this condition is not necessary if we know more about the numerical scheme. For instance, if $K\in\mc T_h$ is a tetrahedron, $\mb V(K)=\mc P_{k+1}(K)^3$ and $Q(K)=\mc P_{k}(K)$, we can choose $\tau_n=0$ and still obtain unique solvability. The proof of this fact will be similar to the proof of the unique solvability of the BDM method.

\subsection{Projections and remainders}
The key in our analysis is finding projections satisfying the following Assumption \ref{asmp:proj}. 
Here, under this general setting, we shall just assume the projection exists and proceed the analysis.
We remark that these projections are not unique in most cases and our target is to find the projections that can well fit the structures of the approximation spaces and therefore give sharp estimates. 
\begin{assmp}[Projection assumption]\label{asmp:proj} For all $K\in\mc T_h$, there exists a projection
\begin{align*}
\Pi_K: H^1(K)^3\times H^1(K)^3\times H^1(K)&\rightarrow W(K)\times V(K)\times Q(K)\\
(\mb w,\mb u,p)&\mapsto (\Pi_K\mb w,\Pi_K\mb u,\Pi_K p),
\end{align*}
such that
\begin{subequations}\label{eq:rq_proj}
\begin{alignat}{5}
\label{eq:rq_proj_1}
&(\Pi_K\mb w-\mb w,\nabla\times\mb v)_K=\dualpr{\mb n\times\mb w-\mr P_N(\mb n\times\mb w),\mb v}_{\pp K} \quad \forall\mb v\in V(K),\\
\label{eq:rq_proj_2}
&(\Pi_K\mb u-\mb u,\mb v)_K=0 \quad \forall\mb v\in\nabla\times W(K)+\nabla Q(K),\\
\label{eq:rq_proj_3}
&(\Pi_K p-p,\nabla\cdot\mb v)_K=0 \quad \forall\mb v\in V(K).
\end{alignat}
\end{subequations}
\end{assmp}
Note that if we have $\mb n_{\pp K}\times V(K)\times\mb n_{\pp K}\subset N(\pp K)$, then
\eqref{eq:rq_proj_1} becomes
\begin{align*}
(\Pi_K\mb w-\mb w,\nabla\times\mb v)_K=0\qquad \forall\mb v\in V(K).
\end{align*}
In this case, Assumption \ref{asmp:proj} holds obviously as a result of \eqref{eq:hdg_sp_1}-\eqref{eq:hdg_sp_3}, since the $L^2$ projection to $W(K)\times V(K)\times Q(K)$ satisfies \eqref{eq:rq_proj}.
In addition, we have used $\Pi_K\mb w$, $\Pi_K u$, and $\Pi_K p$ to represent the first, second, and third component of the projection $\Pi_K$, respectively. Hence $\Pi_K\mb w$ can depend on $\mb u$ and $p$ as well, and this clarification works similarly for $\Pi_K\mb u$ and $\Pi_Kp$.

For all the HDG variants we will study in this paper, Assumption \ref{asmp:proj} is satisfied. Namely, we can explicitly construct  projections that satisfy \eqref{eq:rq_proj_1}-\eqref{eq:rq_proj_3}. We will do this in Section \ref{sec:un_err_ana}. 

We next define two operators associated to the projection $\Pi_K$.

\begin{defi}[Boundary remainders]
For all $K\in\mc T_h$, we define two operators as follows:
\begin{subequations}
\begin{alignat}{5}
\nonumber
\bs\delta_{\pm\tau_t}^{\Pi_K}: H^1(K)^3\times H^1(K)^3&\rightarrow  N(\pp K)\\
\label{eq:def_bdrm_1}
(\mb w,\mb u)&\mapsto\mb n\times\Pi_K\mb w-\mr P_N(\mb n\times\mb w)\pm\tau_t(\mr P_N\Pi_K\mb u-\mr P_N\mb u),\\
\nonumber
\delta_{\pm\tau_n}^{\Pi_K}: H^1(K)^3\times H^1(K)&\rightarrow  M(\pp K)\\
\label{eq:def_bdrm_2}
(\mb u,p)&\mapsto\Pi_K\mb u\cdot\mb n-\mr P_M(\mb u\cdot\mb n)\pm\tau_n(\Pi_K p-\mr P_Mp).
\end{alignat}
\end{subequations}
We call $\bs\delta_{\pm\tau_t}^{\Pi_K}$ the curl-curl boundary remainder and $\delta_{\pm\tau_n}^{\Pi_K}$ the grad-div boundary remainder.
\end{defi}
By \eqref{eq:hdg_sp_4} and \eqref{eq:hdg_sp_5}, 
it is easy to see that the above definition is valid. 
The boundary remainder operators can be regarded as an indicator for how much the projection $\Pi_K$ resembles an HDG projection or a mixed method projection.
Consider the grad-div boundary remainder $\delta_{\tau_n}^{\Pi_K}$.
If we let the second-third component of $\Pi_K$, namely $(\Pi_K\mb u,\Pi_Kp)$, to be replaced by the HDG projection with stabilization function $\tau$, then $\delta_{\tau_n=\tau}^{\Pi_K}=0$ (holds by definition; see \cite{CoGoSa:2010}); if the second-third component is replaced by the Raviart-Thomas (RT) and the Brezzi-Douglas-Marini (BDM) projections \cite{BrDoMa:1985,RaTh:1977} (\cite{Ne:1980,Ne:1986} by N\'{e}d\'{e}lec for $\mbb R^3$ case), we have $\delta_{\tau_n=0}^{\Pi_K}=0$. On the other hand, if the first-second component of $\Pi_K$ is replaced by edge element associated projections ($H(\mr{curl})$ projections \cite{Ne:1980,Ne:1986}), then we have $\bs\delta_{\tau_t=0}^{\Pi_K}=0$.

The following Lemma gives two identities further relating the projection $\Pi_K$ and its associated two boundary remainders.
\begin{lemma}[Weak-commutativity]
For all $K\in\mc T_h$, denote by $\bs\delta_{\pm\tau_t}^{\Pi_K}:=\bs\delta_{\pm\tau_t}^{\Pi_K}(\mb w,\mb u)$
and $\delta_{\pm\tau_n}^{\Pi_K}:=\delta_{\pm\tau_n}^{\Pi_K}(\mb u,p)$ for simplicity.
Then
\begin{subequations}
\label{eq:id_wk_cm}
\begin{alignat}{5}
\label{eq:id_wk_cm_1}
(\nabla\times(\Pi_K\mb w-\mb w),\mb v)_K
\pm\dualpr{\tau_t(\mr P_N\Pi_K\mb u-\mr P_N\mb u),\mb v}_{\pp K}&=\dualpr{\bs\delta_{\pm\tau_t}^{\Pi_K},\mb v}_{\pp K}\quad\forall \mb v\in V(K),\\
\label{eq:id_wk_cm_2}
(\nabla\cdot(\Pi_K\mb u-\mb u),q)_K\pm\dualpr{\tau_n(\Pi_K p-\mr P_Mp),q}_{\pp K}
&=\dualpr{\delta_{\pm\tau_n}^{\Pi_K},q}_{\pp K}\quad\forall q\in Q(K).
\end{alignat}
\end{subequations}
\end{lemma}
\begin{proof}
First note that
\begin{align*}
&(\nabla\times(\Pi_K\mb w-\mb w),\mb v)_K
\pm\dualpr{\tau_t(\mr P_N\Pi_K\mb u-\mr P_N\mb u),\mb v}_{\pp K}\\
&\qquad=\dualpr{\mb n\times(\Pi_K\mb w-\mb w)\pm\tau_t(\mr P_N\Pi_K\mb u-\mr P_N\mb u),\mb v}_{\pp K}
+(\Pi_K\mb w-\mb w,\nabla\times\mb v)_K,
\end{align*}
for all $\mb v\in V(K)$. Equation \eqref{eq:id_wk_cm_1} now follows by using \eqref{eq:rq_proj_1}. Equation \eqref{eq:id_wk_cm_2} can be similarly obtained by using \eqref{eq:rq_proj_2} and \eqref{eq:hdg_sp_5}.
\end{proof}

\subsection{Estimates}\label{sec:general_set_estimate}
\newparagraph{Energy estimates.} To proceed with the analysis,
we assume Assumption \ref{asmp:proj} is satisfied so that we have a projection satisfying \eqref{eq:rq_proj_1}-\eqref{eq:rq_proj_3} for each $K\in\mc T_h$. We next
define the elementwise projections and associated boundary remainders:
\begin{align*}
(\Pi\mb w,\Pi\mb u,\Pi p)=\prod_{K\in\mc T_h}\Pi_K(\mb w,\mb u,p),\quad
\bs\delta_{\tau_t}^\Pi=\prod_{K\in\mc T_h}\bs\delta_{\tau_t}^{\Pi_K}(\mb w,\mb u),\quad
\delta_{\tau_n}^\Pi=\prod_{K\in\mc T_h}\delta_{\tau_n}^{\Pi_K}(\mb u,p).
\end{align*}
We also define the error terms to simplify notation:
\begin{align*}
\bs\varepsilon_h^w=\Pi\mb w-\mb w_h,\ 
\bs\varepsilon_h^u=\Pi\mb u-\mb u_h,\ 
\varepsilon_h^p=\Pi p-p_h,\ 
\widehat{\bs\varepsilon}_h^u=\mr P_N\mb u-\widehat{\mb u}_h,\ 
\widehat{\varepsilon}_h^p=\mr P_Mp - \widehat{p}_h. 
\end{align*}
Note that
\begin{align*}
(\bs\varepsilon_h^w,\bs\varepsilon_h^u,\varepsilon_h^p,\widehat{\bs\varepsilon}_h^u,\widehat{\varepsilon}_h^p)&\in W_h\times V_h\times Q_h\times N_h\times M_h,\\
(\bs\delta_{\tau_t}^\Pi,\delta_{\tau_n}^\Pi)&\in\prod_{K\in\mc T_h}N(\pp K)\times \prod_{K\in\mc T_h}M(\pp K).
\end{align*}

For the following two Propositions (Props. \ref{prop:ene_id} and \ref{prop:duality_id}), we put their proofs in the appendix. Once the HDG variants are specified, we can immediately obtain the $L^2$ error estimates of $\mb w_h$ and $\mb u_h$ by using these two propositions.

\begin{prop}[Energy identity]\label{prop:ene_id}
The following energy identity holds
\begin{align}\label{eq:ene_id}
&(\bs\varepsilon_h^w,\bs\varepsilon_h^w)_{\mc T_h}
+\dualpr{\tau_t(\mr P_N\bs\varepsilon_h^u-\widehat{\bs\varepsilon}_h^u),\mr P_N\bs\varepsilon_h^u-\widehat{\bs\varepsilon}_h^u}_{\pp\mc T_h}
+\dualpr{\tau_n(\varepsilon_h^p-\widehat{\varepsilon}_h^p),\varepsilon_h^p-\widehat{\varepsilon}_h^p}_{\pp\mc T_h}\\
\nonumber
&=(\Pi\mb w-\mb w,\bs\varepsilon_h^w)_{\mc T_h}+\dualpr{\bs\delta_{\tau_t}^\Pi,\mr P_N\bs\varepsilon_h^u-\widehat{\bs\varepsilon}_h^u}_{\pp\mc T_h}+\dualpr{\delta_{\tau_n}^\Pi,\varepsilon_h^p-\widehat{\varepsilon}_h^p}_{\pp\mc T_h}.
\end{align}
\end{prop}

From the above identity, we can obtain an estimate for $\|\mb w_h-\mb w\|_{\mc T_h}$.

\newparagraph{Duality estimates.}
To estimate $\mb u_h$, we consider the dual equations
\begin{subequations}\label{eq:dual_eq}
\begin{alignat}{5}
\mb w^*+\nabla\times\mb u^* &= 0 &\quad& \mr{in}\ \Omega,\\
-\nabla\times\mb w^*-\nabla p^*&= \bs\theta &&\mr{in}\ \Omega,\\
-\nabla\cdot \mb u^* &=0 && \mr{in}\ \Omega,\\
\mb n\times\mb u^*&=\mb 0 &&\mr{on}\ \Gamma,\\
p^* &= 0 &&\mr{on}\ \Gamma.
\end{alignat}
\end{subequations}

\begin{assmp}[Regularity assumption]\label{asmp:ell_reg}
The following inequality holds
\begin{align}\label{eq:ell_reg}
\|\mb w^*\|_{r_1,\Omega}+\|\mb u^*\|_{r_2,\Omega}+\|p^*\|_{r_3,\Omega}
\le C_\mr{reg}\|\bs\theta\|_\Omega,
\end{align}
for any $\bs\theta\in L^2(\Omega)^3$. Here $r_1,r_2\in(1/2,\infty)$, $r_3\ge1$ and $C_\mr{reg}$ is a constant depending only on $\Omega$.
\end{assmp}

Let $\Pi_K^*$ be another projection satisfying Assumption \ref{asmp:proj}. Note that it is allowed to choose $\Pi_K^*=\Pi_K$. Define 
\begin{align*}
&(\Pi^*\mb w^*,\Pi^*\mb u^*,\Pi^* p^*)=\prod_{K\in\mc T_h}\Pi_K^*(\mb w^*,\mb u^*,p^*),\\
&\bs\delta_{-\tau_t}^{\Pi^*}=\prod_{K\in\mc T_h}\bs\delta_{-\tau_t}^{\Pi_K^*}(\mb w^*,\mb u^*),\quad
\delta_{-\tau_n}^{\Pi^*}=\prod_{K\in\mc T_h}\delta_{-\tau_n}^{\Pi_K^*}(\mb u^*,p^*).
\end{align*}

\begin{prop}[Duality identity] \label{prop:duality_id}The following identity holds
\begin{align}\label{eq:duality_id}
&(\Pi\mb w-\mb w,\Pi^*\mb w^*)_{\mc T_h}+\dualpr{\bs\delta_{\tau_t}^\Pi,\Pi^*\mb u^*-\mr P_N\mb u^*}_{\pp\mc T_h}
+\dualpr{\delta_{\tau_n}^\Pi,\Pi^*p^*-\mr P_Mp^*}_{\pp\mc T_h}\\
\nonumber
&=(\Pi^*\mb w^*-\mb w^*,\Pi\mb w-\mb w_h)_{\mc T_h}-\dualpr{\bs\delta_{-\tau_t}^{\Pi^*},\mr P_N\bs\varepsilon_h^u-\widehat{\bs\varepsilon}_h^u}_{\pp\mc T_h}
-\dualpr{\delta_{-\tau_n}^{\Pi^*},\varepsilon_h^p-\widehat{\varepsilon}_h^p}_{\pp\mc T_h}\\
\nonumber
&\quad\,\,+(\bs\theta,\bs\varepsilon_h^u)_{\mc T_h}.
\end{align}
\end{prop}

Let $\bs\theta=\bs\varepsilon_h^u$ and proceed, we can obtain an estimate for $\|\mb u-\mb u_h\|_{\mc T_h}$. We will do this in Section \ref{sec:un_err_ana} when the approximation spaces are specified.

\section{Projections}\label{sec:gal_proj}
In this section, we give a collection of projections which will become the building blocks for constructing projections satisfying Assumption \ref{asmp:proj}. Some of these projections are well known while some are newly devised. 
For those known, we review their constructions and convergence properties. For those new, we prove their optimal convergence under certain shape-regularity conditions of the element. We categorize the projections into two groups: (1) Projections for polyhedral element; (2) Projections for simplex element.

\subsection{Projections for polyhedral element}
In this subsection, we focus on one element $K$, which we assume to be a star-shaped polyhedron (we remark that $K$ is also allowed to be a simplex). 
We define the shape-regularity constant of $K$ as any constant $\gamma>0$ satisfying the following conditions (see \cite{BrSc:2008,DiDr:2017,Ki:2012}):
\begin{itemize}
\item Chunkiness condition. $K$ is star-shaped with respect to a ball with radius $\rho$ and $\frac{h_K}{\rho}\le \gamma$. 
\item Simplex condition. $K$ admits a simplex decomposition such that for any simplex $T$, if $h_T$ is the diameter of $T$ and $\rho_T$ is the inradius, then $\frac{h_T}{\rho_T}\le \gamma$. 
\item Local quasi-uniformity. Let $a^\mr{max}$ and $a^\mr{min}$ be the areas of the largest and smallest face of $K$ respectively, then $\frac{a^\mr{max}}{a^\mr{\min}}\le \gamma$.
\end{itemize}

\newparagraph{$L^2$ projection.} For $k\ge0$, the orthogonal projection (or $L^2$ projection)
\begin{align*}
\Pi_k:L^2(K)^3&\rightarrow\mc P_k(K)^3,\\
\mb u&\mapsto\Pi_k\mb u,
\end{align*}
is defined by solving
\begin{align}\label{eq:eq_L2}
(\Pi_k\mb u-\mb u,\mb v)_K=0\quad\forall \mb v\in\mc P_k(K)^3.
\end{align}
We have (see \cite{DiDr:2017,HoPeSc:2004})
\begin{align}\label{eq:conv_L2}
h_K^{1/2}\|\Pi_k\mb u-\mb u\|_{\pp K}
+\|\Pi_k\mb u-\mb u\|_K\le C h_K^m|\mb u|_{m,K},
\end{align}
where $m\in(1/2,k+1]$ and $C$ depends only on $k$ and the shape-regularity of $K$.

\newparagraph{Curl+ projection.} We denote by $\widetilde{\mc P}_k$ the homogeneous polynomial space of degree $k$ and denote by $\nabla_F$ the surface gradient on face $F$. Define
\begin{align*}
N(\pp K)=\prod_{F\in\mc E_K}\mc P_k(F)^{t}\oplus\nabla_F\widetilde{\mc P}_{k+2}(F),
\end{align*}
and let $\mr P_N: H^{1/2+\epsilon}(K)\rightarrow N(\pp K)$ be the $L^2$ projection to $N(\pp K)$.
For $k\ge 0$, 
the curl+ projection
\begin{align*}
\Pi_k^c:H^{1/2+\epsilon}(K)^3&\rightarrow\mc P_k(K)^3,\\
\mb w&\mapsto\Pi_k^c\mb w,
\end{align*}
is defined by
\begin{subequations}\label{eq:curl+_proj}
\begin{alignat}{5}
\label{eq:curl+_proj_1}
(\Pi_k^c\mb w-\mb w,\mb r)_K&=0 \quad\forall\mb r\in\nabla\times\mc P_k(K)^3\oplus(\nabla\times\mc P_{k+1}(K)^3)^{\perp_k},\\
\label{eq:curl+_proj_2}
(\Pi_k^c\mb w-\mb w,\nabla\times\mb v)_{K}
&=\dualpr{(\mb n\times\mb w)-\mr P_N(\mb n\times\mb w),\mb v}_{\pp K}\\
\nonumber
&\quad\qquad\forall\mb v\in(\mc P_k(K)^3\oplus\nabla\widetilde{\mc P}_{k+2}(K))^{\perp_{k+1}},
\end{alignat}
\end{subequations}
where $\perp_m$ means taking orthogonal complement in $\mc P_m(K)^3$. 

By \eqref{eq:curl+_proj_1} and \eqref{eq:curl+_proj_2}, we obtain
\begin{align}\label{eq:curl+_proj_cor}
(\Pi_k^c\mb w-\mb w,\nabla\times\mb v)_{K}
&=\dualpr{(\mb n\times\mb w)-\mr P_N(\mb n\times\mb w),\mb v}_{\pp K}
\quad\forall\mb v\in\mc P_{k+1}(K)^3.
\end{align}
This can be easily proved by decompose $\mb v=\mb v_1+\mb v_2$, where $\mb v_1\in\mc P_k(K)^3\oplus\nabla\widetilde{\mc P}_{k+2}(K)$ and $\mb v_2\in (\mc P_k(K)^3\oplus\nabla\widetilde{\mc P}_{k+2}(K))^{\perp_{k+1}}$.
In addition, note that if $k\ge1$, then
\begin{align}\label{eq:curl+_0order}
(\Pi_k^c\mb w-\mb w,\mc P_0(K)^3)_K=0,
\end{align}
which can be derived easily from \eqref{eq:curl+_proj_1}.
\begin{theorem}\label{thm:curl+_conv}
The projection $\Pi_k^c$ is well defined and 
\begin{align}\label{eq:curl+_conv}
h_K^{1/2}\|\Pi_k^c\mb w-\mb w\|_{\pp K}+
\|\Pi_k^c\mb w-\mb w\|_K\le C h_K^m|\mb w|_{m,K},
\end{align}
where $m\in(1/2,k+1]$ and $C$ depends only on $k$ and the shape-regularity of $K$.
\end{theorem}
\begin{proof}
See appendix.
\end{proof}

This projection will be the key in our analysis of the two HDG variants using Lehrenfeld-Sch\"{o}berl type stabilization (LS stabilization) function \cite{LeSc:2010, ChQiShSo:2017} (variants $\mathfrak{B+}$ and $\mathfrak{H+}$). 

\subsection{Projections for simplex element}
In this subsection, we focus on one simplex element $K$ in $\mbb R^3$. 

\newparagraph{HDG projection.} Let $\mc R_k(\pp K):=\prod_{F\in\mc E_K}\mc P_k(F)$ to shorten notation. For $k\ge0$, the HDG projection (see \cite{CoGoSa:2010})
\begin{align*}
\Pi_{k,\tau_K}^H:H^1(K)^3\times H^1(K)&\rightarrow\mc P_k(K)^3\times\mc P_k(K),\\
(\mb u,p)&\mapsto (\Pi_{k,\tau_K}^H\mb u,\Pi_{k,\tau_K}^H p),
\end{align*}
is defined by solving
\begin{subequations}\label{eq:def_hdgpr}
\begin{align}
\label{eq:def_hdgpr_1}
(\Pi_{k,\tau_K}^H\mb u-\mb u,\mb r)_K&=0\quad\forall\mb r\in \mc P_{k-1}(K)^3,\\
\label{eq:def_hdgpr_2}
(\Pi_{k,\tau_K}^H p-p,v)_K&=0\quad\forall v\in \mc P_{k-1}(K),\\
\label{eq:def_hdgpr_3}
\dualpr{(\Pi_{k,\tau_K}^H\mb u-\mb u)\cdot\mb n+\tau_K(\Pi_{k,\tau_K}^H p-p),\mu}_{\pp K}&=0\quad
\forall\mu\in\mc R_k(\pp K),
\end{align}
\end{subequations}
where $\tau_K\in\prod_{F\in\mc E_K}\mc P_0(F)$ and it satisfies either $0\neq\tau_K\ge0$
or $0\neq\tau_K\le 0$. 
\begin{theorem}[\cite{CoGoSa:2010}]\label{thm:hdg}
For the projection $\Pi_{k,\tau_K}^H$, we have
\begin{subequations}\label{eq:conv_hdg}
\begin{align}
\label{eq:conv_hdg_1}
h_K^{1/2}\|\mb u-\Pi_{k,\tau_K}^H\mb u\|_{\pp K}
+\|\mb u-\Pi_{k,\tau_K}^H\mb u\|_K&\le C
(h_K^s|\mb u|_{s,K}+\tau_K^\mr{sec}h_K^t|p|_{t,K}),\\
\label{eq:conv_hdg_2}
h_K^{1/2}\|p-\Pi_{k,\tau_K}^H p\|_{\pp K}
+\|p-\Pi_{k,\tau_K}^H p\|_K &\le C (h_K^t|p|_{t,K}+\frac{h_K^s}{\tau_K^\mr{max}}|\div\mb u|_{s-1,K}),
\end{align}
\end{subequations}
with $s,t\in[1,k+1]$,
where $\tau_K^\mr{max}$ and $\tau_K^\mr{sec}$ are the largest and the second largest values of $|\tau_K|$ on the faces of $K$ respectively.
\end{theorem}

\newparagraph{BDM-H projection.} For $k\ge1$, define
\begin{align*}
\Pi_{k,\tau_K}^B:H^1(K)^3&\rightarrow \mc P_k(K)^3,\\
\mb u&\mapsto\Pi_{k,\tau_K}^B\mb u,
\end{align*}
by solving
\begin{subequations}\label{eq:def_bdmh}
\begin{align}
\label{eq:def_bdmh_1}
(\Pi_{k,\tau_K}^B\mb u-\mb u,\mb r)_K &=0\quad\forall\mb r\in \mc N_{k-2}(K),\\
\label{eq:def_bdmh_2}
\dualpr{(\Pi_{k,\tau_K}^B\mb u-\mb u)\cdot\mb n+\tau_K(\Pi_{k-1}p-p),\mu}_{\pp K}&=0\quad\forall\mu\in\mc R_k(\pp K),
\end{align}
\end{subequations}
where $\tau_K\in\prod_{F\in\mc E_K}\mc P_0(F)$ satisfying either $\tau_K\ge0$ or $\tau_K\le0$, and $\mc N_{k-2}(K)$ is the N\'{e}d\'{e}lec space $\mc N_{k-2}(K):=\mc P_{k-2}(K)^d\oplus\{\mb u\in\widetilde{\mc P}_{k-1}(K)^d:\mb u\cdot\mb m=0\}$ with $\mb m=(x,y,z)$.
\begin{prop}\label{prop:bdmh_proj}
For the projection $\Pi_{k,\tau_K}^B$, we have
\begin{align}\label{eq:conv_bdmh}
h_K^{1/2}\|\Pi_{k,\tau_K}^B\mb u-\mb u\|_{\pp K}
+\|\Pi_{k,\tau_K}^B\mb u-\mb u\|_K
\le C\left(h_K^s|\mb u|_{s,K} + \tau_K^\mr{max}h_K^t|p|_{t,K}\right),
\end{align}
where $s\in[1,k+1]$, $t\in[1,k]$, $\tau_K^\mr{max}$ is the largest value of $|\tau_K|$, and $C$ depends only on $k$ and the shape-regularity of $K$. 
\end{prop}
\begin{proof}
See appendix.
\end{proof}

\section{Unified error analysis}\label{sec:un_err_ana}
In this section, we specify those approximations spaces and stabilization functions in the general setting proposed in Section \ref{sec:gen_setting}. Depending on the choices of the approximations spaces and the types of meshes, we construct different projections. 
All these projections satisfy Assumption \ref{asmp:proj} and therefore we can easily obtain the error estimates of $\mb w_h$ and $\mb u_h$ by using Proposition \ref{prop:ene_id} and Proposition \ref{prop:duality_id}.

\subsection{Estimates for solution with high-regularity}\label{sec:sys_err_ana}
In this subsection, we assume $r_1=1$ and $r_2=2$ in Assumption \ref{asmp:ell_reg}.
We remark that this holds if $\Omega$ is assumed to be convex additionally (we have assumed that $\Omega$ is a simply connected polyhedral domain with connected Lipschitz boundary). Its proof can be obtained by using \cite[Theorem 3.5]{GiRa:1986} and then the identity $\nabla\times\nabla\times\mb u=\nabla(\nabla\cdot\mb u)-\Delta\mb u$ to transform the original formulation \eqref{eq:dual_eq} to a Poisson's equation. We will consider four variants (see Table \ref{tb:var_sumry} for an overview) and prove that all the variants are optimal in Theorem \ref{thm:est_variants}. 

We first introduce some notation. Let $\mc T_h^s$ be the collection of all simplex elements in $\mc T_h$ and let $\mc T_h^p=\mc T_h\backslash\mc T_h^s$ be those non-simplex elements. 
We denote by $\pp\mc T_h^*:=\cup_{K\in\mc T_h^*}\{\pp K\}$ with $*\in\{s,p\}$ as the collections of the boundaries of the simplex and the non-simplex elements respectively.

\newparagraph{Variant $\mathfrak{B}$:
$W\times V\times Q\times N\times M=
\mc P_{k}^3\times\mc P_{k+1}^3\times\mc P_{k}\times\mc P_{k+1}^t\times\mc P_{k+1}$.}
To the best of our knowledge, this variant has not been considered before. We require $k\ge0$.

Let $c_1,c_2>0$ be two fixed constants.
For each $K\in\mc T_h^s$, we choose the stabilization functions such that $\tau_n\big|_{\pp K}\le c_1h_K$ and $c_1h_K^{-1}\le \tau_t\big|_{\pp K}\le c_2h_K^{-1}$. We choose the projection
\begin{subequations}\label{eq:vb_s}
\begin{align}
\label{eq:vb_s_1}
\Pi_K(\mb w,\mb u,p)&:=(\Pi_{k}\mb w,\Pi_{k+1,\tau_n}^B\mb u,\Pi_{k} p),\\
\label{eq:vb_s_2}
\Pi_K^*(\mb w^*,\mb u^*,p^*)&:=(\Pi_{k}\mb w^*,\Pi_{k+1,-\tau_n}^B\mb u^*,\Pi_{k} p^*).
\end{align}
\end{subequations}
For each $K\in\mc T_h^p$, we choose the stabilization functions such that $c_1h_K\le \tau_n\big|_{\pp K}\le c_2h_K$ and $c_1h_K^{-1}\le \tau_t\big|_{\pp K}\le c_2h_K^{-1}$. We choose the projection
\begin{subequations}\label{eq:vb_ns}
\begin{align}
\label{eq:vb_ns_1}
\Pi_K(\mb w,\mb u,p)&:=(\Pi_{k}\mb w,\Pi_{k+1}\mb u,\Pi_{k} p),\\
\label{eq:vb_ns_2}
\Pi_K^*(\mb w^*,\mb u^*,p^*)&:=(\Pi_{k}\mb w^*,\Pi_{k+1}\mb u^*,\Pi_{k} p^*).
\end{align}
\end{subequations}
It is easy to verify that the projections $\Pi_K$ and $\Pi_K^*$ satisfy Assumption \ref{asmp:proj} for all $K\in\mc T_h$ by using \eqref{eq:def_bdmh_1} and \eqref{eq:eq_L2}.

\newparagraph{Variant $\mathfrak{H}$: 
$W\times V\times Q\times N\times M=
\mc P_{k}^3\times\mc P_{k+1}^3\times\mc P_{k+1}\times\mc P_{k+1}^t\times\mc P_{k+1}$.}
This variant has been considered in \cite{ChCuXu:2019}, where the scheme is shown to be optimal if all $K\in\mc T_h$ are simplex and the stabilization functions satisfy $\tau_t\big|_{\pp K}\approx h_K^{-1}$ and $\tau_n\big|_{\pp K}\approx h_K$. We here prove that the scheme is actually optimal for polyhedral elements. For simplex elements, we show that a weaker condition on the stabilization function $(\tau_n\big|_{\pp K})^\mr{sec}\lesssim h_K$ can provide optimal convergence, where $(\cdot\big|_{\pp K})^\mr{sec}$ represents the second largest value of $|\tau_n|$ on the faces of $K$. We require $k\ge0$.

For each $K\in\mc T_h^s$, we choose the stabilization functions such that $\tau_n\neq0$, $(\tau_n\big|_{\pp K})^\mr{sec}\le c_1h_K$ and $c_1h_K^{-1}\le \tau_t\big|_{\pp K}\le c_2h_K^{-1}$. We choose the projection
\begin{subequations}\label{eq:vh_s}
\begin{align}
\label{eq:vh_s_1}
\Pi_K(\mb w,\mb u,p)&:=(\Pi_{k}\mb w,\Pi_{k+1,\tau_n}^H\mb u,\Pi_{k+1,\tau_n}^H p),\\
\label{eq:vh_s_2}
\Pi_K^*(\mb w^*,\mb u^*,p^*)&:=(\Pi_{k}\mb w^*,\Pi_{k+1,-\tau_n}^H\mb u^*,\Pi_{k+1,-\tau_n}^H p^*).
\end{align}
\end{subequations}
For each $K\in\mc T_h^p$, we choose the stabilization functions such that $c_1h_K\le \tau_n\big|_{\pp K}\le c_2h_K$ and $c_1h_K^{-1}\le \tau_t\big|_{\pp K}\le c_2h_K^{-1}$. We choose the projection
\begin{subequations}\label{eq:vh_ns}
\begin{align}
\label{eq:vh_ns_1}
\Pi_K(\mb w,\mb u,p)&:=(\Pi_{k}\mb w,\Pi_{k+1}\mb u,\Pi_{k+1} p),\\
\label{eq:vh_ns_2}
\Pi_K^*(\mb w^*,\mb u^*,p^*)&:=(\Pi_{k}\mb w^*,\Pi_{k+1}\mb u^*,\Pi_{k+1} p^*).
\end{align}
\end{subequations}
It is easy to verify that the projections $\Pi_K$ and $\Pi_K^*$ satisfy Assumption \ref{asmp:proj} for all $K\in\mc T_h$ by using \eqref{eq:eq_L2}, \eqref{eq:def_hdgpr_1} and \eqref{eq:def_hdgpr_2}.

\newparagraph{Variant $\mathfrak{B+}$:
$W\times V\times Q\times N\times M=
\mc P_k^3\times\mc P_{k+1}^3\times\mc P_{k}\times\mc P_k^t\oplus\nabla_F\widetilde{\mc P}_{k+2}\times\mc P_{k+1}$.}
This variant has been analyzed in \cite{ChQiShSo:2017}. Compared to \cite{ChQiShSo:2017}, we give estimates in a slightly more general setting where the stabilization functions are allowed to be chosen more freely depending on the types of elements (simplex or not simplex). We require $k\ge1$.

For each $K\in\mc T_h^s$, we choose the stabilization functions such that $\tau_n\big|_{\pp K}\le c_1h_K$ and $c_1h_K^{-1}\le \tau_t\big|_{\pp K}\le c_2h_K^{-1}$. We choose the projection
\begin{subequations}\label{eq:vb+_s}
\begin{align}
\label{eq:vb+_s_1}
\Pi_K(\mb w,\mb u,p)&:=(\Pi_k^c\mb w,\Pi_{k+1,\tau_n}^B\mb u,\Pi_{k} p),\\
\label{eq:vb+_s_2}
\Pi_K^*(\mb w^*,\mb u^*,p^*)&:=(\Pi_k^c\mb w^*,\Pi_{k+1,-\tau_n}^B\mb u^*,\Pi_k p^*).
\end{align}
\end{subequations}
For each $K\in\mc T_h^p$, we choose the stabilization functions such that $c_1h_K\le \tau_n\big|_{\pp K}\le c_2h_K$ and $c_1h_K^{-1}\le \tau_t\big|_{\pp K}\le c_2h_K^{-1}$. We choose the projection
\begin{subequations}\label{eq:vb+_ns}
\begin{align}
\label{eq:vb+_ns_1}
\Pi_K(\mb w,\mb u,p)&:=(\Pi_k^c\mb w,\Pi_{k+1}\mb u,\Pi_k p),\\
\label{eq:vb+_ns_2}
\Pi_K^*(\mb w^*,\mb u^*,p^*)&:=(\Pi_k^c\mb w^*,\Pi_{k+1}\mb u^*,\Pi_k p^*).
\end{align}
\end{subequations}
We can verify that the projections $\Pi_K$ and $\Pi_K^*$ satisfy Assumption \ref{asmp:proj} for all $K\in\mc T_h$ by using \eqref{eq:eq_L2}, \eqref{eq:curl+_proj_cor}, and \eqref{eq:def_bdmh_1}.

\newparagraph{Variant $\mathfrak{H+}$:
$W\times V\times Q\times N\times M=
\mc P_k^3\times\mc P_{k+1}^3\times\mc P_{k+1}\times\mc P_k^t\oplus\nabla_F\widetilde{\mc P}_{k+2}\times\mc P_{k+1}$.}
\label{sec:HDG++}
To the best of our knowledge, this variant has not been considered before. We require $k\ge1$.

For each $K\in\mc T_h^s$, we choose the stabilization functions such that $\tau_n\neq0$, $(\tau_n\big|_{\pp K})^\mr{sec}\le c_1h_K$ and $c_1h_K^{-1}\le \tau_t\big|_{\pp K}\le c_2h_K^{-1}$. We choose the projection
\begin{subequations}\label{eq:vh+_s}
\begin{align}
\label{eq:vh+_s_1}
\Pi_K(\mb w,\mb u,p)&:=(\Pi_k^c\mb w,\Pi_{k+1,\tau_n}^H\mb u,\Pi_{k+1,\tau_n}^H p),\\
\label{eq:vh+_s_2}
\Pi_K^*(\mb w^*,\mb u^*,p^*)&:=(\Pi_k^c\mb w^*,\Pi_{k+1,-\tau_n}^H\mb u^*,\Pi_{k+1,-\tau_n}^H p^*).
\end{align}
\end{subequations}
For each $K\in\mc T_h^p$, we choose the stabilization functions such that $c_1h_K\le \tau_n\big|_{\pp K}\le c_2h_K$ and $c_1h_K^{-1}\le \tau_t\big|_{\pp K}\le c_2h_K^{-1}$. We choose the projection
\begin{subequations}\label{eq:vh+_ns}
\begin{align}
\label{eq:vh+_ns_1}
\Pi_K(\mb w,\mb u,p)&:=(\Pi_k^c\mb w,\Pi_{k+1}\mb u,\Pi_{k+1} p),\\
\label{eq:vh+_ns_2}
\Pi_K^*(\mb w^*,\mb u^*,p^*)&:=(\Pi_k^c\mb w^*,\Pi_{k+1}\mb u^*,\Pi_{k+1} p^*).
\end{align}
\end{subequations}
We can verify that the projections $\Pi_K$ and $\Pi_K^*$ satisfy Assumption \ref{asmp:proj} for all $K\in\mc T_h$ by using \eqref{eq:eq_L2}, \eqref{eq:curl+_proj_cor},  \eqref{eq:def_hdgpr_1} and \eqref{eq:def_hdgpr_2}.

The following estimates hold simultaneously for all the four variants $\mathfrak{B}$, $\mathfrak{H}$, $\mathfrak{B+}$, $\mathfrak{H+}$.

\begin{prop}[estimates of boundary remainders]
If $K\in\mc T_h^s$, then
\begin{align}\label{eq:bdtn00}
\delta_{\tau_n}^{\Pi_K}=\delta_{-\tau_n}^{\Pi_K^*}=0.
\end{align}
If $K\in\mc T_h^p$, then
\begin{subequations}\label{eq:bdtn}
\begin{align}
\label{eq:bdtn_1}
\|\tau_n^{-1/2}\delta_{\tau_n}^{\Pi_K}\|_{\pp K}
&\le C\left(
h_K^{s-1}|\mb u|_{s,K}+ h_K^{t}|p|_{t,K}\right),\\
\label{eq:bdtn_2}
\|\tau_n^{-1/2}\delta_{-\tau_n}^{\Pi_K^*}\|_{\pp K}
&\le C\left(
h_K^{s-1}|\mb u^*|_{s,K}+ h_K^{t}|p^*|_{t,K}\right).
\end{align}
\end{subequations}
For all $K\in\mc T_h$, we have
\begin{subequations}\label{eq:bdtt}
\begin{align}
\label{eq:bdtt_1}
\|\tau_t^{-1/2}\bs\delta_{\tau_t}^{\Pi_K}\|_{\pp K}
&\le C\left(
h_K^m|\mb w|_{m,K}+h_K^{s-1}|\mb u|_{s,K}+ h_K^{t}|p|_{t,K}\right),\\
\label{eq:bdtt_2}
\|\tau_t^{-1/2}\bs\delta_{-\tau_t}^{\Pi_K^*}\|_{\pp K}
&\le C\left(
h_K^m|\mb w^*|_{m,K}+h_K^{s-1}|\mb u^*|_{s,K}+ h_K^{t}|p^*|_{t,K}\right).
\end{align}
\end{subequations}
In the above estimates, $m\in[1,k+1]$ and $s,t\in[1,k+2]$ for variants $\mathfrak{H}$ and $\mathfrak{H+}$; $m,t\in[1,k+1]$ and $s\in[1,k+2]$ for variants $\mathfrak{B}$ and $\mathfrak{B+}$. The constant $C$ depends only on $k$, $c_1$, $c_2$, and the shape-regularity of $K$.
\end{prop}

\begin{proof}
{\bf Step 1--estimates about $\delta_{\tau_n}^{\Pi_K}$ and $\delta_{-\tau_n}^{\Pi_K^*}$.}
First note that for the all variants, $\mr P_M$ has fixed meaning - the $L^2$ projection to $M(\pp K):=\prod_{F\in\mc E_K}N(F)=\prod_{F\in\mc E_K}\mc P_{k+1}(F)$.
We next show that $\delta_{\tau_n}^{\Pi_K}=\delta_{-\tau_n}^{\Pi_K^*}=0$ if $K\in\mc T_h^s$. 
For variants $\mathfrak{B}$ and $\mathfrak{B+}$, by \eqref{eq:def_bdrm_2}, \eqref{eq:vb_s_1} and \eqref{eq:vb+_s_1} we have
\begin{align*}
\delta_{\tau_n}^{\Pi_K} = (\Pi_{k+1,\tau_n}^B\mb u)\cdot\mb n-\mr P_M(\mb u\cdot\mb n)+\tau_n(\Pi_kp-\mr P_Mp).
\end{align*}
Now by \eqref{eq:def_bdmh_2} we have $\delta_{\tau_n}^{\Pi_K}=0$. By \eqref{eq:vb_s_2} and \eqref{eq:vb+_s_2}, we can similarly obtain $\delta_{-\tau_n}^{\Pi_K^*}=0$. For variants $\mathfrak{H}$ and $\mathfrak{H+}$, by \eqref{eq:def_bdrm_2}, \eqref{eq:vh_s_1} and \eqref{eq:vh+_s_1} we have
\begin{align*}
\delta_{\tau_n}^{\Pi_K} = (\Pi_{k+1,\tau_n}^H\mb u)\cdot\mb n-\mr P_M(\mb u\cdot\mb n)+\tau_n(\Pi_{k+1,\tau_n}^Hp-\mr P_Mp).
\end{align*}
Hence $\delta_{\tau_n}^{\Pi_K}=0$ by \eqref{eq:def_hdgpr_3}. By \eqref{eq:vh_s_2} and \eqref{eq:vh+_s_2}, we can similarly obtain $\delta_{-\tau_n}^{\Pi_K^*}=0$.

Consider $K\in\mc T_h^p$. By \eqref{eq:def_bdrm_2}, \eqref{eq:vb_ns_1}, \eqref{eq:vh_ns_1}, \eqref{eq:vb+_ns_1} and \eqref{eq:vh+_ns_1}, we have
\begin{align*}
\tau_n^{-1/2}\delta_{\tau_n}^{\Pi_K}=
\left\{
\begin{array}{ll}
\tau_n^{-1/2}(\Pi_{k+1}\mb u\cdot\mb n-\mr P_M(\mb u\cdot\mb n))+\tau_n^{1/2}(\Pi_{k}p-\mr P_Mp)&\mr{for}\ \mathfrak{B,B+},\\[5pt]
\tau_n^{-1/2}(\Pi_{k+1}\mb u\cdot\mb n-\mr P_M(\mb u\cdot\mb n))+\tau_n^{1/2}(\Pi_{k+1}p-\mr P_Mp)&\mr{for}\ \mathfrak{H,H+}.
\end{array}
\right.
\end{align*}
The above with the fact that $\tau_n\big|_{\pp K}\approx h_K$ for all $K\in\mc T_h^p$ implies \eqref{eq:bdtn_1}. We can similarly obtain \eqref{eq:bdtn_2} by \eqref{eq:def_bdrm_2}, \eqref{eq:vb_ns_2}, \eqref{eq:vh_ns_2}, \eqref{eq:vb+_ns_2} and \eqref{eq:vh+_ns_2}.\\

\noindent{\bf Step 2--estimates concerning $\bs\delta_{\tau_t}^{\Pi_K}$ and $\bs\delta_{-\tau_t}^{\Pi_K^*}$.}\\
First note that $N(\pp K)=\prod_{F\in\mc E_K}\mc P_{k+1}(F)^t$ for variants $\mathfrak{B}$ and $\mathfrak{H}$ while $N(\pp K)=\prod_{F\in\mc E_K}\mc P_{k}(F)^t\oplus\nabla_F\widetilde{\mc P}_{k+2}(F)$ for variants $\mathfrak{B+}$ and $\mathfrak{H+}$. Hence the projection $\mr P_N$ --defined as the $L^2$ projection to $N(\pp K)$-- will change its meaning accordingly depending on which variant is considered.
Now, by \eqref{eq:def_bdrm_1}, \eqref{eq:vb_s_1}, \eqref{eq:vb_ns_1}, \eqref{eq:vh_s_1}, \eqref{eq:vh_ns_1}, \eqref{eq:vb+_s_1}, \eqref{eq:vb+_ns_1}, \eqref{eq:vh+_s_1} and \eqref{eq:vh+_ns_1}, we have
\begin{align*}
\tau_t^{-1/2}\bs\delta_{\tau_t}^{\Pi_K}
=\tau_t^{-1/2}\mr P_NT_1+\tau_t^{1/2}\mr P_NT_2,
\end{align*}
where $T_1^K$ and $T_2^K$ are defined by the values in Table \ref{tb:val_t12_bd}.

\begin{table}[ht]
\centering
\begin{tabular}{|c|c|c|c|c|}
\hline
 & \multicolumn{2}{c|}{Simplex $K$} & \multicolumn{2}{c|}{Non-simplex $K$}\\
\hline
Variant & $T_1^K$ & $T_2^K$ & $T_1^K$ & $T_2^K$\\
\hline
$\mathfrak{B}$& $\mb n\times\Pi_k\mb w-\mb n\times\mb w$ & $\Pi_{k+1,\tau_n}^B\mb u-\mb u$ & $\mb n\times\Pi_k\mb w-\mb n\times\mb w$ & $\Pi_{k+1}\mb u-\mb u$\\
$\mathfrak{H}$ & $\mb n\times\Pi_k\mb w-\mb n\times\mb w$ &  $\Pi_{k+1,\tau_n}^H\mb u-\mb u$ & $\mb n\times\Pi_k\mb w-\mb n\times\mb w$ & $\Pi_{k+1}\mb u-\mb u$\\
$\mathfrak{B+}$ & $\mb n\times\Pi_k^c\mb w-\mb n\times\mb w$ &$\Pi_{k+1,\tau_n}^B\mb u-\mb u$& $\mb n\times\Pi_k^c\mb w-\mb n\times\mb w$ & $\Pi_{k+1}\mb u-\mb u$ \\
$\mathfrak{H+}$ &$\mb n\times\Pi_k^c\mb w-\mb n\times\mb w$&$\Pi_{k+1,\tau_n}^H\mb u-\mb u$&$\mb n\times\Pi_k^c\mb w-\mb n\times\mb w$&$\Pi_{k+1}\mb u-\mb u$\\
\hline
\end{tabular}
\caption{Values of $T_1^K$ and $T_2^K$ for variants $\mathfrak{B}$, $\mathfrak{H}$, $\mathfrak{B+}$, $\mathfrak{H+}$ on simplex elements and non-simplex elements.}
\label{tb:val_t12_bd}
\end{table}

Recall that $\tau_t\big|_{\pp K}\approx h_K^{-1}$ for all elements $K\in\mc T_h$ and all variants.
From Table \ref{tb:val_t12_bd} and by \eqref{eq:conv_L2} and \eqref{eq:curl+_conv}, we obtain that
\begin{align*}
\|T_1^K\|_{\pp K}\lesssim h_K^{m-1/2}|\mb w|_{m,K}\qquad\forall m\in[1,k+1].
\end{align*}
Also from Table \ref{tb:val_t12_bd}, and by \eqref{eq:conv_L2}, \eqref{eq:conv_bdmh} and \eqref{eq:conv_hdg_1}, we have
\begin{align*}
\|T_2^K\|_{\pp K}\lesssim h_K^{s-1/2}|\mb u|_{s,K}+h_K^{t+1/2}|p|_{t,K},
\end{align*}
where $s\in[1,k+2]$ for all variants; $t\in[1,k+1]$ for variants $\mathfrak{B}$ and $\mathfrak{B+}$,  and $t\in[1,k+2]$ for variants $\mathfrak{H}$ and $\mathfrak{H+}$.
We can give a similar estimate to $\bs\delta_{-\tau_t}^{\Pi_K^*}$ by using \eqref{eq:def_bdrm_1}
 \eqref{eq:vb_s_2}, \eqref{eq:vb_ns_2}, \eqref{eq:vh_s_2}, \eqref{eq:vh_ns_2}, \eqref{eq:vb+_s_2}, \eqref{eq:vb+_ns_2}, \eqref{eq:vh+_s_2} and \eqref{eq:vh+_ns_2}. This completes the proof.
\end{proof}

\begin{theorem}[$L^2$ estimates of $\mb w_h$ and $\mb u_h$]\label{thm:est_variants}
We have
\begin{align}\label{eq:est_wh}
&\|\mb w-\mb w_h\|_{\mc T_h}
+\|\tau_t^{1/2}\mr P_N(\bs\varepsilon_h^u-\widehat{\bs\varepsilon}_h^u)\|_{\pp\mc T_h}
+\|\tau_n^{1/2}(\varepsilon_h^p-\widehat{\varepsilon}_h^p)\|_{\pp\mc T_h}\\
\nonumber
&\qquad\le C_1\left( h^m|\mb w|_{m,\Omega}+h^{s-1}|\mb u|_{s,\Omega}+h^t|p|_{t,\Omega}\right).
\end{align}
If the regularity condition \eqref{eq:ell_reg} holds (with $r_1=1$ and $r_2=2$), then we have
\begin{align}\label{eq:est_uh}
\|\mb u-\mb u_h\|_{\mc T_h}\le C_2\left( h^{m+1}|\mb w|_{m,\Omega}+h^{s}|\mb u|_{s,\Omega}+h^{t+1}|p|_{t,\Omega}\right).
\end{align}
Here, $m\in[1,k+1]$ and $s,t\in[1,k+2]$ for variants $\mathfrak{H}$ and $\mathfrak{H+}$; $m,t\in[1,k+1]$ and $s\in[1,k+2]$ for variants $\mathfrak{B}$ and $\mathfrak{B+}$. The constant $C_1$ depends only on $k$, $c_1$, $c_2$ and the shape-regularity of $\mc T_h$ while $C_2$ depends additionally on $C_\mr{reg}$.
\end{theorem}

\begin{proof}
By \eqref{eq:ene_id} and \eqref{eq:bdtn00} we have
\begin{align*}
&\|\bs\varepsilon_h^w\|_{\mc T_h}
+\|\tau_t^{1/2}(\mr P_N\bs\varepsilon_h^u-\widehat{\bs\varepsilon}_h^u)\|_{\pp\mc T_h}
+\|\tau_n^{1/2}(\varepsilon_h^p-\widehat{\varepsilon}_h^p)\|_{\pp\mc T_h}\\
&\quad\lesssim \|\Pi\mb w-\mb w\|_{\mc T_h}+\|\tau_t^{-1/2}\bs\delta_{\tau_t}^\Pi\|_{\pp\mc T_h}+\|\tau_n^{-1/2}\delta_{\tau_n}^\Pi\|_{\pp\mc T_h^p}.
\end{align*}
Note that $\Pi\mb w=\Pi_k\mb w$ for variants $\mathfrak{B}$ and $\mathfrak{H}$, and $\Pi\mb w=\Pi_k^c\mb w$ for variants $\mathfrak{B+}$ and $\mathfrak{H+}$.
We next use \eqref{eq:bdtn_1}, \eqref{eq:bdtt_1}, \eqref{eq:conv_L2} and \eqref{eq:curl+_conv}. Then \eqref{eq:est_wh} is obtained.

We next consider \eqref{eq:est_uh}. Let $\bs\theta=\bs\varepsilon_h^u$ in the dual equations \eqref{eq:dual_eq}. By \eqref{eq:duality_id}, \eqref{eq:curl+_0order} and \eqref{eq:bdtn00} we have
\begin{align}
\label{eq:prf_7}
\|\bs\varepsilon_h^u\|_{\mc T_h}^2&\le
\|\Pi^*\mb w^*-\mb w^*\|_{\mc T_h}\|\mb w_h-\mb w\|_{\mc T_h}+\|\Pi\mb w-\mb w\|_{\mc T_h}\|\mb w^*-\Pi_0\mb w^*\|_{\mc T_h}\\
\nonumber
&\quad+\|\tau_t^{-1/2}\bs\delta_{\tau_t}^\Pi\|_{\pp\mc T_h}\|\tau_t^{1/2}(\Pi^*\mb u^*-\mr P_N\mb u^*)\|_{\pp\mc T_h}\\
\nonumber
&\quad+\|\tau_n^{-1/2}\delta_{\tau_n}^\Pi\|_{\pp\mc T_h^p}\|\tau_n^{1/2}(\Pi^*p^*-\mr P_Mp^*)\|_{\pp\mc T_h^p}\\
\nonumber
&\quad+\|\tau_t^{-1/2}\bs\delta_{-\tau_t}^{\Pi^*}\|_{\pp\mc T_h}\|\tau_t^{1/2}(\mr P_N\bs\varepsilon_h^u-\widehat{\bs\varepsilon}_h^u)\|_{\pp\mc T_h}\\
\nonumber
&\quad+\|\tau_n^{-1/2}\delta_{-\tau_n}^{\Pi^*}\|_{\pp\mc T_h^p}
\|\tau_n^{1/2}(\varepsilon_h^p-\widehat{\varepsilon}_h^p)\|_{\pp\mc T_h^p}.
\end{align}
Note that in the above inequality, we have $\Pi^*\mb w^*=\Pi_k\mb w^*$ for variants $\mathfrak{B}$ and $\mathfrak{H}$, and $\Pi^*\mb w^*=\Pi_k^c\mb w^*$ for variants $\mathfrak{B+}$ and $\mathfrak{H+}$. Therefore, by \eqref{eq:conv_L2} and \eqref{eq:curl+_conv} we have
\begin{align*}
\|\Pi^*\mb w^*-\mb w^*\|_{\mc T_h}+\|\mb w^*-\Pi_0\mb w^*\|_{\mc T_h}
\lesssim h\|\mb w^*\|_{1,\Omega}.
\end{align*}
Recall that for all variants $\tau_t\big|_K\approx h_K^{-1}$ for all $K\in\mc T_h$ and $\tau_n\big|_{\pp K}\approx h_K$ for all $K\in\mc T_h^p$.
These with \eqref{eq:conv_L2}, \eqref{eq:conv_hdg_1} and \eqref{eq:conv_bdmh} imply
\begin{align*}
\|\tau_t^{1/2}(\Pi^*\mb u^*-\mr P_N\mb u^*)\|_{\pp\mc T_h}
+\|\tau_n^{1/2}(\Pi^*p^*-\mr P_Mp^*)\|_{\pp\mc T_h^p}
\lesssim h(\|\mb u^*\|_{2,\Omega}+\|p^*\|_{1,\Omega}).
\end{align*}
By \eqref{eq:bdtn_2} and \eqref{eq:bdtt_2} we have
\begin{align*}
\|\tau_t^{-1/2}\bs\delta_{-\tau_t}^{\Pi^*}\|_{\pp\mc T_h}
+\|\tau_n^{-1/2}\delta_{-\tau_n}^{\Pi^*}\|_{\pp\mc T_h^p}
\lesssim h(\|\mb w^*\|_{1,\Omega}+\|\mb u^*\|_{2,\Omega}+\|p^*\|_{1,\Omega}).
\end{align*}
Therefore, by the regularity assumption \eqref{eq:ell_reg} (with $r_1=1$ and $r_2=2$), we have
\begin{align*}
\|\bs\varepsilon_h^u\|_{\mc T_h}
&\lesssim h\big(
\|\mb w_h-\mb w\|_{\mc T_h}+\|\Pi\mb w-\mb w\|_{\mc T_h}
+\|\tau_t^{-1/2}\bs\delta_{\tau_t}^\Pi\|_{\pp\mc T_h}
+\|\tau_n^{-1/2}\delta_{\tau_n}^\Pi\|_{\pp\mc T_h^p}\\
&\quad +\|\tau_t^{1/2}(\mr P_N\bs\varepsilon_h^u-\widehat{\bs\varepsilon}_h^u)\|_{\pp\mc T_h}
+\|\tau_n^{1/2}(\varepsilon_h^p-\widehat{\varepsilon}_h^p)\|_{\pp\mc T_h^p}\big).
\end{align*}
Combing the above with \eqref{eq:est_wh}, \eqref{eq:bdtn_1}, \eqref{eq:bdtt_1}, \eqref{eq:conv_L2} and \eqref{eq:curl+_conv}, we obtain \eqref{eq:est_uh}. This completes the proof. 
\end{proof}

Now we summarize the results obtained in this subsection. Table \ref{tb:var_sumry} gives an overview of the choices of the approximation spaces and stabilization functions for the four variants we have analyzed.
\begin{table}[ht]
\centering
\begin{tabular}{|c|c|c|c|c|c|c|c|}
\hline
Variant & $k$ & $Q$ & $N$ & $\tau_n$\\
\hline
$\mathfrak{B}$ & $k\ge0$ & $\mc P_{k}$ & $\mc P_{k+1}^t$ & $
\begin{array}{cc}
\lesssim h_K & K\in\mc T_h^s\\
\approx h_K & K\in\mc T_h^p
\end{array}$\\
\hline
$\mathfrak{H}$ & $k\ge0$ & $\mc P_{k+1}$ & $\mc P_{k+1}^t$ & $
\begin{array}{cc}
0\neq\tau_n,\ \tau_n^\mr{sec}\lesssim h_K & K\in\mc T_h^s\\
\approx h_K & K\in\mc T_h^p
\end{array}$ \\
\hline
$\mathfrak{B+}$ & $k\ge1$ & $\mc P_{k}$ & $\mc P_{k}^t\oplus\nabla_F\widetilde{\mc P}_{k+2}$ & $
\begin{array}{cc}
\lesssim h_K & K\in\mc T_h^s\\
\approx h_K & K\in\mc T_h^p
\end{array}$ \\
\hline
$\mathfrak{H+}$ & $k\ge1$ & $\mc P_{k+1}$ & $\mc P_{k}^t\oplus\nabla_F\widetilde{\mc P}_{k+2}$ & $
\begin{array}{cc}
0\neq\tau_n,\ \tau_n^\mr{sec}\lesssim h_K & K\in\mc T_h^s\\
\approx h_K & K\in\mc T_h^p
\end{array}$\\
\hline
\end{tabular}
\caption{Approximation spaces and stabilization functions of variants $\mathfrak{B}$, $\mathfrak{H}$, $\mathfrak{B+}$, $\mathfrak{H+}$. For all the four variants, $W=\mc P_k^3$, $V=\mc P_{k+1}^3$, $M=\mc P_{k+1}$, and $\tau_t\approx h_K^{-1}$.}
\label{tb:var_sumry}
\end{table}

We have proved that for all the four HDG variants, $\mb w_h$ and $\mb u_h$ are optimally convergent in $L^2$ norms. From Table \ref{tb:var_sumry}, we observe that the variants $\mathfrak{B+}$ and $\mathfrak{H+}$, compared with $\mathfrak{B}$ and $\mathfrak{H}$, use smaller trace spaces $N$ while achieve the same rate of convergence. Actually, by \eqref{eq:est_wh},
we have $\|\tau_t^{1/2}\widehat{\varepsilon}_h^u\|_{\pp\mc T_h}\lesssim h^{k+1}+\|\tau_t^{1/2}\mr P_N\varepsilon_h^u\|_{\pp\mc T_h}$ (for smooth enough exact solutions). Assuming $\mc T_h$ is quasi-uniform for simplicity, and then using \eqref{eq:est_uh} and the fact that $\tau_t\approx h^{-1}$, we have
$\|\tau_t^{1/2}\mr P_N\varepsilon_h^u\|_{\pp\mc T_h}\lesssim h^{k+1}$. Therefore
\begin{align}
\|\mr P_N\mb u-\widehat{\mb u}_h\|_h&:=\left(\sum_{K\in\mc T_h}
\|h_K^{1/2}(\mr P_N\mb u-\widehat{\mb u}_h)\|_{\pp K}^2\right)^{1/2}\\
&\approx \|\tau_t^{-1/2}\widehat{\bs\varepsilon}_h^u\|_{\pp\mc T_h}
\approx h\|\tau_t^{1/2}\widehat{\bs\varepsilon}_h^u\|_{\pp\mc T_h}\lesssim h^{k+2}.
\end{align}
Since $N(F)$ is a proper subspace of $\mc P_{k+1}(F)^t$ for variants $\mathfrak{B+}$ and $\mathfrak{H+}$, we say $\mr P_N\mb u-\widehat{\mb u}_h$ achieves superconvergence.
These superconvergence properties are due to the LS stabilization functions \cite{LeSc:2010, ChQiShSo:2017} used in their formulations. Correspondingly, our analysis of variants $\mathfrak{B+}$ and $\mathfrak{H+}$ involve using the projection defined by \eqref{eq:curl+_proj}, which we construct especially for this situation. We also observe that for variants $\mathfrak{H}$ and $\mathfrak{H+}$, only the second largest values of $\tau_n$ on the four faces of the simplex elements affect the convergence. This suggests that we can send one face value of $\tau_n$ to infinity for simplex elements $K$ and this will have no effect on the convergence of $\mb w_h$ and $\mb u_h$. This feature holds as a result of the convergence properties of the HDG projection \eqref{eq:conv_hdg_1} and we will verify this feature in the numerical experiments.

\subsection{Estimates for solution with low-regularity}\label{sec:low_reg}
In this subsection, we study the convergence of the four variants for solution with low-regularity ($\mb w,\mb u\in H^{s}(\mc T_h)^3$ with $s\in(1/2,1]$). To compare, we also include the standard HDG method \cite{NgPeCo2:2011}, namely,
 $W\times V\times Q\times N\times M=\mc P_{k}^3\times\mc P_{k}^3\times\mc P_{k}\times\mc P_{k}^t\times\mc P_{k}$. For the four variants, we choose their projections as defined in \eqref{eq:vb_ns}, \eqref{eq:vh_ns}, \eqref{eq:vb+_ns},  \eqref{eq:vh+_ns}. For the standard HDG method, we simply use $L^2$ projections:
\begin{subequations}
\begin{align}
\Pi_K(\mb w,\mb u,p)&:=(\Pi_k\mb w,\Pi_{k}\mb u,\Pi_k p),\\
\Pi_K^*(\mb w^*,\mb u^*,p^*)&:=(\Pi_k\mb w^*,\Pi_k\mb u^*,\Pi_k p^*).
\end{align}
\end{subequations}
Apparently these projections satisfy Assumption \ref{asmp:proj}. 
We let the stabilization functions to be $\tau_t=h_K^\beta$ and $\tau_n=h_K^\alpha$ so that their effect on the convergence can be better observed. We remark that all estimates in this subsection are for general polyhedral meshes. For more specific types of meshes (tetrahedral meshes for instance), it is possible to construct more tailored projections, instead of $L^2$ projections, to obtain sharper estimates. This will be the aim of future work.

The following theorem holds for all the four variants and the standard HDG method.

\begin{theorem}\label{thm:low_reg}
We have
\begin{align}\label{eq:conv_lowreg}
\|\mb w_h-\mb w\|_{\mc T_h}\le C T_1,\quad
\|\mb u_h-\mb u\|_{\mc T_h}\le C T_2T_1+h^{s_u}|\mb u|_{s_u,\mc T_h},
\end{align}
where
\begin{align*}
T_1&:=h^{s_w-1/2}\max\{h^{1/2},h^{-\beta/2}\}|\mb w|_{s_w,\mc T_h}
+h^{s_u-1/2}\max\{h^{\beta/2},h^{-\alpha/2}\}|\mb u|_{s_u,\mc T_h}\\
&\quad\ +h^{1/2+\alpha/2}|p|_{1,\mc T_h},\\
T_2&:=h^{r_2-1/2}\max\{h^{\beta/2},h^{-\alpha/2}\}
+h^{\alpha/2+1/2}+h^{r_1-1/2}\max\{h^{1/2},h^{-\beta/2}\}.
\end{align*}
Here, $s_w,s_u\in(1/2,1]$, the index $r_1,r_2$ appear in \eqref{eq:ell_reg}, and $C$ is independent of $h$.
\end{theorem}
\begin{proof}
First note that we have only used the $L^2$ projection and the projection defined by \eqref{eq:curl+_proj} for the five methods (the four variants and the standard HDG method).
By \eqref{eq:conv_L2}, \eqref{eq:curl+_conv}, \eqref{eq:def_bdrm_1} and \eqref{eq:def_bdrm_2}, we have
\begin{align}
\|\tau_n^{-1/2}\delta_{\tau_n}^{\Pi_K}\|_{\pp K}
\lesssim h_K^{s_u-1/2-\alpha/2}|\mb u|_{s_u,K}
+h_K^{s_p-1/2+\alpha/2}|p|_{s_p,K}\\
\|\tau_t^{-1/2}\bs\delta_{\tau_t}^{\Pi_K}\|_{\pp K}
\lesssim h_K^{s_w-1/2-\beta/2}|\mb w|_{s_w,K}
+h_K^{s_u-1/2+\beta/2}|\mb u|_{s_u,K},
\end{align}
where $s_w,s_u,s_p\in(1/2,1]$.
Now by \eqref{eq:ene_id} we have
\begin{align*}
\|\mb w_h-\mb w\|_{\mc T_h}\lesssim \|\bs\Pi\mb w-\mb w\|_{\mc T_h}+\|\tau_n^{-1/2}\delta_{\tau_n}^\Pi\|_{\pp\mc T_h}+\|\tau_t^{-1/2}\bs\delta_{\tau_t}^\Pi\|_{\pp\mc T_h},
\end{align*}
which gives the estimate for $\|\mb w_h-\mb w\|_{\mc T_h}$.

On the other hand, by \eqref{eq:duality_id} we obtain an inequality similar to \eqref{eq:prf_7} (replacing all $\mc T_h^p$ by $\mc T_h$), from which, we have
\begin{align*}
\|\bs\varepsilon_h^u\|_{\mc T_h}^2&\lesssim T_1\bigg(h^{r_1-1/2}\max\{h^{1/2},h^{-\beta/2}\}|\mb w^*|_{r_1,\Omega}
+h^{\alpha/2+1/2}|p^*|_{1,\Omega}\\
&\quad +h^{r_2-1/2}\max\{h^{\beta/2},h^{-\alpha/2}\}|\mb u^*|_{r_2,\Omega}
\bigg)\lesssim T_1T_2\|\bs\varepsilon_h^u\|_{\mc T_h},
\end{align*}
where \eqref{eq:ell_reg} is used for the last inequality sign. This completes the proof.
\end{proof}

Suppose $\mb f$ is divergence-free, then $p=0$. 
If we choose $\alpha=\beta=0$, then \eqref{eq:conv_lowreg} gives
\begin{subequations}\label{eq:abeq1_conv}
\begin{align}
\|\mb w_h-\mb w\|_{\mc T_h} 
&\lesssim h^{s_w-1/2}|\mb w|_{s_w,\mc T_h}
+h^{s_u-1/2}|\mb u|_{s_u,\mc T_h},\\
\|\mb u_h-\mb u\|_{\mc T_h}&
\lesssim \max\{h^{r_2-1/2},h^{r_1-1/2},h^{1/2}\}\left(h^{s_w-1/2}|\mb w|_{s_w,\mc T_h}
+h^{s_u-1/2}|\mb u|_{s_u,\mc T_h}\right)\\
\nonumber
&\quad +h^{s_u}|\mb u|_{s_u,\mc T_h}.
\end{align}
\end{subequations}
In this case, $\mb w_h$ and $\mb u_h$ converge but not necessary optimally. In addition, $\mb u_h$ converges faster than $\mb w_h$ with an additional order, which depends on the values of $r_1$ and $r_2$ (see \eqref{asmp:ell_reg}). For instance, if $r_1,r_2\ge1$, then $\mb u_h$ converges optimally. On the other hand, if we choose $\alpha=-1$, then
\begin{align}\label{eq:am1_conv}
\|\mb w_h-\mb w\|_{\mc T_h}\lesssim h^{s_w-1/2}\max\{h^{1/2},h^{-\beta/2}\}|\mb w|_{s_w,\mc T_h}
+h^{s_u-1/2}\max\{h^{\beta/2},h^{1/2}\}|\mb u|_{s_u,\mc T_h}.
\end{align}
Therefore, depending on the a priori information about $\mb w$ and $\mb u$, we can adjust $\tau_t=h^\beta$ to achieve a better convergence rate. For instance, if $\mb w=0$, then we can choose $\beta=1$ and then $\|\mb w_h-\mb w\|_{\mc T_h}\lesssim h^{s_u}$.

\section{Numerical tests}\label{sec:num_exp}
\subsection{Solution with high-regularity}
In this subsection, we provide some numerical experiments for variant $\mathfrak{H+}$ and variant $\mathfrak{B}$ for smooth exact solution. Note that the corresponding experiments for variants $\mathfrak{H}$ and $\mathfrak{B+}$ have appeared in \cite{ChCuXu:2019,ChQiShSo:2017}. We consider a cubic domain $\Omega=[0,1]^3$ uniformly discretized by tetrahedral elements and choose the exact solutions as the following:
\begin{align*}
\mb u(x,y,z)&=(\sin(\pi x)\sin(\pi y)\sin(\pi z),\cos(\pi x)\cos(\pi y)\sin(\pi z),x^5+y^5),\\
p(x,y,z)&=\sin(\pi x)\sin(\pi y)\sin(\pi z),
\end{align*}
where $\mb w$ and the data $\mb f,\mb g$ are chosen such that \eqref{eq:PDE} are satisfied.

{\bf Tests for variant $\mathfrak{H+}$.}
We conduct three error tests (denoted by A,B and C). For Test A, we choose $\tau_t\big|_{\pp K}=h_K^{-1}$ and $\tau_n\big|_{\pp K}=h_K$. For Test B, we choose the same value of $\tau_t$ as Test A, but we set $\tau_n$ on one face of $K$ to be $\frac{10^5}{h_K^2}$ and the rest to be $0$. Note that both the choices of the stabilization functions for Test A and B satisfy the requirement of variant $\mathfrak{H+}$ (see Table \ref{tb:var_sumry}). We finally consider Test C, where we choose $\tau_t\big|_{\pp K}=h_K^{-1}$ and $\tau_n\big|_{\pp K}=\frac{10^5}{h_K^2}$. This choice of $\tau_n$ violates the requirement of variant $\mathfrak{H+}$ (see Table \ref{tb:var_sumry}). 

From Table \ref{tb:test_1} and \ref{tb:test_2}, we observe that both $\mb w_h$ and $\mb u_h$ converge at optimal order for Test A and B. We also observe that the discrete solutions in Test B converge slightly faster than those in Test A. This is consistent with our analysis (we remark that the choice of the stabilization functions of Test B minimizes the HDG projection errors (see \eqref{eq:conv_hdg}) compared to Test A). From Table \ref{tb:test_3}, we observe that the discrete solutions in Test C lose the optimal convergence rate. This to some degree supports the sharpness of our estimates .

\begin{table}[ht]
\centering
\begin{tabular}{|c|c|c|c|c|c|}
\hline
k & h & \multicolumn{2}{c|}{$\|\mb w_h-\mb w\|_{\mc T_h}$} & \multicolumn{2}{c|}{$\|\mb u_h-\mb u\|_{\mc T_h}$} \\
\hline
  && Error & Order                & Error &  Order\\
\hline
 1 & 1.41e+00  &  1.76e+00  &  -  &   2.21e+00  &  -\\
   & 7.07e-01  &  5.38e-01  &  1.71  &   3.42e-01  &  2.69\\
   & 3.54e-01  &  1.36e-01  &  1.98  &   4.49e-02  &  2.93\\
   & 1.77e-01  &  3.49e-02  &  1.96  &   5.88e-03  &  2.93\\
\hline
 2 & 1.41e+00  &  7.30e-01  &  -  &   1.11e+00  &  -\\
   & 7.07e-01  &  1.10e-01  &  2.73  &   8.49e-02  &  3.71\\
   & 3.54e-01  &  1.50e-02  &  2.87  &   5.26e-03  &  4.01\\
   & 1.77e-01  &  1.96e-03  &  2.94  &   3.06e-04  &  4.10\\
\hline
 3 & 1.41e+00  &  2.50e-01  &  -  &   3.91e-01  &  -\\ 
   & 7.07e-01  &  2.27e-02  &  3.46  &   1.82e-02  &  4.43\\
   & 3.54e-01  &  1.62e-03  &  3.81  &   5.95e-04  &  4.93\\
   & 1.77e-01  &  1.07e-04  &  3.92  &   1.92e-05  &  4.95\\
\hline
\end{tabular}
\caption{Test A: $\tau_t\big|_{\pp K}=h_K^{-1}$, $\tau_n\big|_{\pp K}=h_K$.}
\label{tb:test_1}
\end{table}

\begin{table}[ht]
\centering
\begin{tabular}{|c|c|c|c|c|c|}
\hline
k & h & \multicolumn{2}{c|}{$\|\mb w_h-\mb w\|_{\mc T_h}$} & \multicolumn{2}{c|}{$\|\mb u_h-\mb u\|_{\mc T_h}$} \\
\hline
  && Error & Order                & Error &  Order\\
\hline
1& 1.41e+00  &  1.77e+00  &  -  &   2.02e+00  &  -\\
 & 7.07e-01  &  5.37e-01  &  1.72  &   2.85e-01  &  2.83\\
 & 3.54e-01  &  1.36e-01  &  1.99  &   3.67e-02  &  2.95\\
 & 1.77e-01  &  3.47e-02  &  1.96  &   4.91e-03  &  2.90\\
\hline
2& 1.41e+00  &  7.43e-01  &  -  &   9.80e-01  &  -\\
 & 7.07e-01  &  1.09e-01  &  2.76  &   6.72e-02  &  3.87\\
 & 3.54e-01  &  1.48e-02  &  2.89  &   4.01e-03  &  4.07\\
 & 1.77e-01  &  1.92e-03  &  2.95  &   2.22e-04  &  4.17\\
\hline
3& 1.41e+00  &  2.53e-01  &  -  &   3.44e-01  &  -\\
 & 7.07e-01  &  2.26e-02  &  3.49  &   1.43e-02  &  4.58\\
 & 3.54e-01  &  1.59e-03  &  3.83  &   4.43e-04  &  5.02\\
 & 1.77e-01  &  1.04e-04  &  3.93  &   1.42e-05  &  4.96\\
\hline
\end{tabular}
\caption{Test B: $\tau_t\big|_{\pp K}=h_K^{-1}$, $(\tau_n\big|_{\pp K})^\mr{max}=\frac{10^5}{h_K^2}$ and $(\tau_n\big|_{\pp K})^\mr{sec}=0$. Recall that we denote by  $(\tau_n\big|_{\pp K})^\mr{max}$ and $(\tau_n\big|_{\pp K})^\mr{sec}$ the largest and the second largest values of $\tau_n$ on $\pp K$ respectively.}
\label{tb:test_2}
\end{table}

\begin{table}[ht]
\centering
\begin{tabular}{|c|c|c|c|c|c|}
\hline
k & h & \multicolumn{2}{c|}{$\|\mb w_h-\mb w\|_{\mc T_h}$} & \multicolumn{2}{c|}{$\|\mb u_h-\mb u\|_{\mc T_h}$} \\
\hline
  && Error & Order                & Error &  Order\\
\hline
 1 & 1.41e+00  &  1.76e+00  &  -  &   2.10e+04  &   -\\
   & 7.07e-01  &  5.42e-01  &  1.70  &   4.07e+04  &  -0.96\\
   & 3.54e-01  &  1.38e-01  &  1.98  &   1.58e+04  &   1.37\\
   & 1.77e-01  &  3.51e-02  &  1.97  &   4.31e+03  &   1.88\\
\hline
 2 & 1.41e+00  &  7.18e-01  &  -  &   8.71e+03  &  -\\
   & 7.07e-01  &  1.11e-01  &  2.69  &   2.85e+03  &  1.61\\
   & 3.54e-01  &  1.52e-02  &  2.88  &   5.13e+02  &  2.48\\
\hline
 3 & 1.41e+00  &  2.51e-01  &  -  &   1.25e+03  &  -\\
   & 7.07e-01  &  2.29e-02  &  3.46  &   3.01e+02  &  2.06\\
   & 3.54e-01  &  1.64e-03  &  3.80  &   1.98e+01  &  3.93\\
\hline
\end{tabular}
\caption{Test C: $\tau_t\big|_{\pp K}=h_K^{-1}$, $\tau_n\big|_{\pp K}=\frac{10^5}{h_K^2}$.}
\label{tb:test_3}
\end{table}

{\bf Tests for variant $\mathfrak{B}$.}
We test two cases for variant $\mathfrak{B}$ (denoted by Test D and E). For Test D, we choose $\tau_t\big|_{\pp K}=h_K^{-1}$ and $\tau_n\big|_{\pp K}=h_K$. For Test E, we choose $\tau_t\big|_{\pp K}=h_K^{-1}$ and $\tau_n\big|_{\pp K}=0$. Both cases satisfy the requirements of the stabilization functions for variant $\mathfrak{B}$ (see Table \ref{tb:var_sumry}). From Table \ref{tb:test_4} and \ref{tb:test_5}, we observe optimal convergence rate of $\mb w_h$ and $\mb u_h$ in both tests. We also observe that the numerical solutions in Test E converge slightly faster than those in Test D. This is consistent with our analysis (notice that by \eqref{eq:conv_bdmh}, we know the choice of the stabilization functions in Test E minimizes the BDM-H projection errors compared to Test D).

\begin{table}[ht]
\centering
\begin{tabular}{|c|c|c|c|c|c|}
\hline
k & h & \multicolumn{2}{c|}{$\|\mb w_h-\mb w\|_{\mc T_h}$} & \multicolumn{2}{c|}{$\|\mb u_h-\mb u\|_{\mc T_h}$} \\
\hline
  && Error & Order                & Error &  Order\\
\hline
 0 & 1.41e+00  &  2.86e+00  &  -  &   3.30e+00  &  -\\
   & 7.07e-01  &  2.07e+00  &  0.47  &   1.35e+00  &  1.29\\
   & 3.54e-01  &  1.09e+00  &  0.92  &   3.60e-01  &  1.90\\
   & 1.77e-01  &  5.22e-01  &  1.07  &   8.92e-02  &  2.01\\
\hline
 1 & 1.41e+00  &  1.73e+00  &  -  &   2.16e+00  &  -\\
   & 7.07e-01  &  5.36e-01  &  1.69  &   3.75e-01  &  2.52\\
   & 3.54e-01  &  1.37e-01  &  1.97  &   5.01e-02  &  2.90\\
   & 1.77e-01  &  3.53e-02  &  1.96  &   6.43e-03  &  2.96\\
\hline
 2 & 1.41e+00  &  7.27e-01  &  -  &   1.14e+00  &  -\\
   & 7.07e-01  &  1.10e-01  &  2.72  &   9.27e-02  &  3.63\\
   & 3.54e-01  &  1.51e-02  &  2.87  &   6.06e-03  &  3.93\\
   & 1.77e-01  &  1.97e-03  &  2.94  &   3.73e-04  &  4.02\\
\hline
\end{tabular}
\caption{Test D: $\tau_t\big|_{\pp K}=h_K^{-1}$, $\tau_n\big|_{\pp K}=h_K$.}
\label{tb:test_4}
\end{table}

\begin{table}[ht]
\centering
\begin{tabular}{|c|c|c|c|c|c|}
\hline
k & h & \multicolumn{2}{c|}{$\|\mb w_h-\mb w\|_{\mc T_h}$} & \multicolumn{2}{c|}{$\|\mb u_h-\mb u\|_{\mc T_h}$} \\
\hline
  && Error & Order                & Error &  Order\\
\hline
 0 & 1.41e+00  &  2.88e+00  &  -  &   2.96e+00  &  -\\
   & 7.07e-01  &  2.05e+00  &  0.49  &   1.01e+00  &  1.56\\ 
   & 3.54e-01  &  1.07e+00  &  0.94  &   2.32e-01  &  2.12\\ 
   & 1.77e-01  &  5.03e-01  &  1.09  &   5.07e-02  &  2.19\\ 
\hline
 1 & 1.41e+00  &  1.74e+00  &  -  &   1.82e+00  &  -\\
   & 7.07e-01  &  5.36e-01  &  1.70  &   2.73e-01  &  2.74\\
   & 3.54e-01  &  1.36e-01  &  1.98  &   3.52e-02  &  2.95\\
   & 1.77e-01  &  3.50e-02  &  1.96  &   4.70e-03  &  2.90\\
\hline
 2 & 1.41e+00  &  7.44e-01  &  -  &   9.26e-01  &  -\\
   & 7.07e-01  &  1.09e-01  &  2.77  &   6.47e-02  &  3.84\\
   & 3.54e-01  &  1.47e-02  &  2.90  &   3.87e-03  &  4.06\\
   & 1.77e-01  &  1.89e-03  &  2.96  &   2.15e-04  &  4.17\\
\hline
\end{tabular}
\caption{Test E: $\tau_t\big|_{\pp K}=h_K^{-1}$, $\tau_n\big|_{\pp K}=0$.}
\label{tb:test_5}
\end{table}

\subsection{Solution with low-regularity}\label{sec:num_lowreg} In this subsection, we consider an L-shape domain 
\begin{align*}
\Omega=([-1,1]^2\backslash [-1,0]^2)\times[0,1],
\end{align*}
uniformly discretized by tetrahedral elements. We use the following exact solution:
\begin{align*}
\mb u(x,y,z)=(\partial_x (r^{2/3}\sin(\frac{2}{3}\theta)),\partial_y(r^{2/3}\sin(\frac{2}{3}\theta)),0),\quad
p(x,y,z)\equiv 0,
\end{align*}
where $\mb w$ and the data $\mb f,\mb g$ are chosen such that \eqref{eq:PDE} are satisfied. Note that $\mb u\in H^{2/3}(\Omega)$ and $\mb w=\nabla\times\mb u=0$. Similar experiment settings have appeared in \cite{ChQiSh:2018,HoPeSc:2004}.

We test the convergence of $\mb w_h$ and $\mb u_h$ in $L^2$ norms for the standard HDG method (Table \ref{tb:low_reg_sd}), the variant $\mathfrak{B}$ (Table \ref{tb:low_reg_vb}) and the variant $\mathfrak{H}+$ (Table \ref{tb:low_reg_vh}). For the first set of tests, we choose $\tau_t=\tau_n=1$ ($\alpha=\beta=0$). By \eqref{eq:abeq1_conv}, we expect $\mb w_h$ to converge at least in $\mc O(h^{2/3-1/2})$ and $\mb u_h$ to converge faster with an additional order. From Table \ref{tb:low_reg_sd}, \ref{tb:low_reg_vb} and \ref{tb:low_reg_vh}, we observe that $\mb u_h$ converges faster than $\mb w_h$, which is consistent with the analysis. We also observe that $\mb w_h$ converges faster than the estimate $\mc O(h^{2/3-1/2})$ and $\mb u_h$ converges almost in optimal order. 
For the second set of tests, we modify the stabilization functions to $\tau_t=h_K$ and $\tau_n=\frac{1}{h_K}$. 
Since $\mb w=0$, by \eqref{eq:am1_conv}, $\mb w_h$ should converge at least in order $\mc O(h^{2/3})$. From Table \ref{tb:low_reg_sd}, \ref{tb:low_reg_vb} and \ref{tb:low_reg_vh}, we observe that $\mb w_h$ converges in about $\mc O(h)$, which is faster than the first set of tests when $\alpha=\beta=0$. This agrees with our analysis.

\begin{table}[ht]
\centering
\begin{tabular}{|c|c|c|c|c|c|}
\hline
$\tau$ & h & \multicolumn{2}{c|}{$\|\mb w_h-\mb w\|_{\mc T_h}$} & \multicolumn{2}{c|}{$\|\mb u_h-\mb u\|_{\mc T_h}$} \\
\hline
  && Error & Order                & Error &  Order\\
\hline
$\tau_t=1,\ \tau_n=1$&1.41e+00  &  1.84e-01  &  -  &   2.82e-01  &  -\\
  &  7.07e-01 &   1.60e-01  &  0.20  &   1.91e-01  &  0.56\\ 
  &  3.54e-01 &   1.31e-01  &  0.29  &   1.25e-01  &  0.61\\ 
  &  1.77e-01 &   1.02e-01  &  0.36  &   8.10e-02  &  0.63\\ 
  &  8.84e-02 &   7.85e-02  &  0.38  &   5.22e-02  &  0.63\\
\hline
$\tau_t=h_K,\ \tau_n=\frac{1}{h_K}$&1.41e+00  &  2.24e-01  &  -  &   2.83e-01  &  -\\
&  7.07e-01  &  1.25e-01  &  0.85   &  1.89e-01  &  0.58\\
  &  3.54e-01  &  5.38e-02  &  1.21   &  1.23e-01  &  0.62\\ 
  &  1.77e-01  &  2.14e-02  &  1.33   &  7.91e-02  &  0.64\\ 
  &  8.84e-02  &  1.00e-02  &  1.10   &  5.04e-02  &  0.65\\
\hline
\end{tabular}
\caption{Convergence of the standard HDG ($k=0$) for solution with low-regularity.}
\label{tb:low_reg_sd}
\end{table}

\begin{table}[ht]
\centering
\begin{tabular}{|c|c|c|c|c|c|}
\hline
$\tau$ & h & \multicolumn{2}{c|}{$\|\mb w_h-\mb w\|_{\mc T_h}$} & \multicolumn{2}{c|}{$\|\mb u_h-\mb u\|_{\mc T_h}$} \\
\hline
  && Error & Order                & Error &  Order\\
\hline
$\tau_t=1,\ \tau_n=1$ & 1.41e+00  &  6.07e-02  &  -  &   1.61e-01  &  -\\
  &  7.07e-01  &  5.03e-02  &  0.27  &   1.03e-01  &  0.63\\ 
  &  3.54e-01  &  3.63e-02  &  0.47  &   6.62e-02  &  0.64\\ 
  &  1.77e-01  &  2.55e-02  &  0.51  &   4.20e-02  &  0.65\\
\hline
$\tau_t=h_K,\ \tau_n=\frac{1}{h_K}$ & 1.41e+00  &  7.60e-02  &  -  &   1.61e-01  &  -\\
  &  7.07e-01  &  3.79e-02  &  1.01  &   1.03e-01  &  0.64\\
  &  3.54e-01  &  1.52e-02  &  1.32  &   6.61e-02  &  0.64\\ 
  &  1.77e-01  &  8.05e-03  &  0.92  &   4.20e-02  &  0.65\\
\hline
\end{tabular}
\caption{Convergence of variant $\mathfrak{B}$ ($k=0$) for solution with low-regularity.}
\label{tb:low_reg_vb}
\end{table}

\begin{table}[ht]
\centering
\begin{tabular}{|c|c|c|c|c|c|}
\hline
$\tau$ & h & \multicolumn{2}{c|}{$\|\mb w_h-\mb w\|_{\mc T_h}$} & \multicolumn{2}{c|}{$\|\mb u_h-\mb u\|_{\mc T_h}$} \\
\hline
  && Error & Order                & Error &  Order\\
\hline
$\tau_t=1,\ \tau_n=1$&1.41e+00  &  7.79e-02  &  -   &  1.68e-01  &  -\\
  &  7.07e-01  &  6.48e-02  &  0.27  &   1.06e-01  &  0.66\\
  &  3.54e-01  &  4.66e-02  &  0.48  &   6.74e-02  &  0.66\\ 
  &  1.77e-01  &  3.28e-02  &  0.51  &   4.26e-02  &  0.66\\
\hline
$\tau_t=h_K,\ \tau_n=\frac{1}{h_K}$&1.41e+00  &  9.81e-02  &  -  &   1.69e-01  &  -\\
  &  7.07e-01  &  4.87e-02  &  1.01  &   1.06e-01  &  0.68\\
  &  3.54e-01  &  1.91e-02  &  1.35  &   6.74e-02  &  0.65\\
  &  1.77e-01  &  9.44e-03  &  1.01  &   4.27e-02  &  0.66\\
\hline
\end{tabular}
\caption{Convergence of variant $\mathfrak{H}+$ ($k=0$) for solution with low-regularity.}
\label{tb:low_reg_vh}
\end{table}

\section*{Conclusions}
We have proposed a framework that enables us to analyze different variants of HDG methods for the static Maxwell equations in one analysis. The analysis is as simple and concise as the well known projection-based error analysis of the mixed finite element and the HDG methods, while more general, thanks to the introduction of the boundary remainders.
We use the framework to analyze four variants of HDG methods. For the two known variants $\mathfrak{B+}$ and $\mathfrak{H}$, we recover the existing optimal estimates and relax the conditions on the types of meshes and stabilization functions. We also propose two new variants $\mathfrak{B}$ and $\mathfrak{H+}$ and compare these four variants. For solution with low-regularity, we give an analysis to the four variants and the standard HDG method on general polyhedral meshes, and investigate the effect of different stabilization functions. The numerical experiments are consistent our estimates.

Note that we have assumed the constant permittivity and permeability, which is apparently a simplification of the real case. The main reason for this simplification is that we have used various types of projections as tools of our analysis. However, many of these projections were constructed for elliptic diffusion and they assume the solution satisfying at least $H^1$ regularity. But for Maxwell's equations with non-constant material parameters, we can only expect the solution to belong to the space $H^s(\Omega)$ where $s$ can be even less than $\frac{1}{2}$ \cite{BoGuLu:2013}. Therefore, generalizing the analysis to the case of non-constant material parameters require first restudying these projections for functions with low-regularity. In Section \ref{sec:low_reg}, we have briefly explored the application of the unified analysis framework to the low-regularity regime. Unlike the analysis in the high-regularity regime in Section \ref{sec:sys_err_ana} where schemes-tailored projections have been used, we have mostly used $L^2$ projections in Section \ref{sec:low_reg} since their properties are more well understood for solution with low-regularity. Another benefit of using $L^2$ projections is that they allow the analysis to hold for general polyhedral meshes. However, a major disadvantage is that $L^2$ projections do not use any specific structures of the approximation spaces or the stabilization functions (compared with N\'{e}d\'{e}lec, RT, BDM, or HDG projections). Therefore, they often do not provide the sharpest estimates. Naturally, studying schemes-tailored projections that work in the low-regularity regime constitutes one aspect future work. 

From the energy identity \eqref{eq:ene_id}, we obtain that the convergence of $\mb w_h$ is controlled by the convergence of the projection $\bs\Pi\mb w$ and the convergence of the two boundary remainders, namely $\bs\delta_{\tau_t}^\Pi$ and $\delta_{\tau_n}^\Pi$. We have explored the direction of choosing suitable projections such that $\delta_{\tau_n}^{\Pi}=0$. This allows us to have a quite flexible choice of $\tau_n$, as is demonstrated in Section \ref{sec:sys_err_ana}; see also Table \ref{tb:var_sumry}. More importantly, a vanishing boundary remainder excludes the case of suboptimal convergence. This corresponds to the effort of $M$-decompositions. However, we have not much explored the direction of choosing projections such that $\bs\delta_{\tau_t}=0$ but this seems necessary when the solution has low-regularity.
When the solution has high regularity (Section \ref{sec:sys_err_ana}), it is possible to adjust the stabilization function (for instance, we have chosen $\tau_t=h_K^{-1}$) to achieve optimal convergence (variant $\mathfrak{B}$ and $\mathfrak{H}$) and even super-convergence (variant $\mathfrak{B}+$ and $\mathfrak{H}+$) by using the LS stabilization functions. However, when the solution has only low-regularity, this approach does not seem to work. To more clearly deliver this observation, we refer again to the energy identity \eqref{eq:ene_id}, from which we can control the error by a summation of terms, among which a vital one is
\begin{align*}
\tau_t^{-1/2}\bs\delta_{\tau_t}^\Pi=
\tau_t^{-1/2}(\mb n\times\bs\Pi\mb w-\mb P_N(\mb n\times\mb w))+\tau_t^{1/2}(\mb P_N\bs\Pi\mb u-\mb P_N\mb u).
\end{align*}
If $\mb u$ has higher regularity than $\mb w$ (for instance $\mb u\in H^{s+1}(\Omega)^3$ and $\mb w\in H^s(\Omega)^3$), we can increase $\tau_t$ (for instance, $\tau_t=h_K^{-1}$) and obtain an optimal estimate for $\mb w_h$. Similarly, if $\mb w$ has higher regularity, we can decrease $\tau_t$ to achieve a faster convergence (for instance, see the arguments following \eqref{eq:am1_conv} and the corresponding numerical experiments in Section \ref{sec:num_lowreg}). However, when $\mb w$ and $\mb u$ have similar regularity, say $H^s(\mc T_h)^3$, then the optimal choice of $\tau_t$ is $1$, and we can only conclude that this remainder term converges in $\mc O(h^{s-1/2})$, {\sl unless} we can prove the existence of a projection rendering $\bs\delta_{\tau_t}^\Pi=0$ or at least $\bs\delta_{\tau_t}^\Pi=\mc O(h^s)$. Thus, for the cases when $\mb w$ and $\mb u$ have similar low regularity, it seems necessary to seek suitable combinations of the space triplet $W$-$V$-$N$ such that there exists a projection $\Pi^\mr{curl}$ rendering $\bs\delta_{\tau_t}^{\Pi^\mr{curl}}$ vanishing or small enough. This constitutes another aspect of future work.

\appendix
\section{Proofs}
\subsection{Proofs in Section \ref{sec:general_set_estimate}}
We here aim to prove Proposition \ref{prop:ene_id} and Proposition \ref{prop:duality_id}. We begin by proving the following three lemmas. 

\begin{lemma}
We have
\begin{subequations}\label{eq:proj_eq}
\begin{alignat}{5}
\label{eq:proj_eq_1}
(\Pi\mb w,\mb r)_{\mc T_h}-(\Pi\mb u,\nabla\times\mb r)_{\mc T_h}-\dualpr{\mr P_N\mb u,\mb r\times\mb n}_{\pp\mc T_h}&=(\Pi\mb w-\mb w,\mb r)_{\mc T_h},\\
\label{eq:proj_eq_2}
(\nabla\times\Pi\mb w,\mb v)_{\mc T_h}+\dualpr{\tau_t\mr P_N(\Pi\mb u-\mr P_N\mb u),\mb v}_{\pp\mc T_h}&\\
\nonumber
-(\Pi p,\nabla\cdot\mb v)_{\mc T_h}+\dualpr{\mr P_Mp,\mb v\cdot\mb n}_{\pp\mc T_h}
&=(\mb f,\mb v)_{\mc T_h}+\dualpr{\bs\delta_{\tau_t}^\Pi,\mb v}_{\pp\mc T_h},\\
\label{eq:proj_eq_3}
(\nabla\cdot\Pi\mb u,q)_{\mc T_h}+\dualpr{\tau_n(\Pi p-\mr P_Mp),q}_{\pp\mc T_h}
&=\dualpr{\delta_{\tau_n}^\Pi,q}_{\pp\mc T_h},\\
\label{eq:proj_eq_4}
-\dualpr{\mb n\times\Pi\mb w+\tau_t(\Pi\mb u-\mr P_N\mb u),\bs\eta}_{\pp\mc T_h\backslash\Gamma}&=-\dualpr{\bs\delta_{\tau_t}^\Pi,\bs\eta}_{\pp\mc T_h\backslash\Gamma},\\
\label{eq:proj_eq_5}
-\dualpr{\mr P_N\mb u,\bs\eta}_\Gamma &= -\dualpr{\mb g\times\mb n,\bs\eta}_\Gamma,\\
\label{eq:proj_eq_6}
-\dualpr{\Pi\mb u\cdot\mb n+\tau_n(\Pi p-\mr P_Mp),\mu}_{\pp\mc T_h\backslash\Gamma}
&=-\dualpr{\delta_{\tau_n}^\Pi,\mu}_{\pp\mc T_h\backslash\Gamma},\\
\label{eq:proj_eq_7}
-\dualpr{\mr P_Mp,\mu}_\Gamma&=0,
\end{alignat}
\end{subequations}
for all $(\mb r,\mb v,q,\bs\eta,\mu)
\in W_h\times V_h\times Q_h\times N_h\times M_h$.
\end{lemma}
\begin{proof}
Equation \eqref{eq:proj_eq_1} holds as a result of \eqref{eq:rq_proj_2} and \eqref{eq:hdg_sp_4}. We obtain \eqref{eq:proj_eq_2} by using \eqref{eq:id_wk_cm_1}, \eqref{eq:rq_proj_3} and \eqref{eq:hdg_sp_5}. We obtain \eqref{eq:proj_eq_3} by using \eqref{eq:id_wk_cm_2}. Equations \eqref{eq:proj_eq_4} and \eqref{eq:proj_eq_6} hold by the definitions of the two boundary remainders \eqref{eq:def_bdrm_1} and \eqref{eq:def_bdrm_2}, and also \eqref{eq:hdg_sp_4} and \eqref{eq:hdg_sp_5}. Equations \eqref{eq:proj_eq_5} and \eqref{eq:proj_eq_7} are obviously true.
\end{proof}

\begin{lemma}
The following error equations hold
\begin{subequations}
\label{eq:err_eqns}
\begin{alignat}{5}
\label{eq:err_eqns_1}
(\bs\varepsilon_h^w,\mb r)_{\mc T_h}-(\bs\varepsilon_h^u,\nabla\times\mb r)_{\mc T_h}-\dualpr{\widehat{\bs\varepsilon}_h^u,\mb r\times\mb n}_{\pp\mc T_h}&=(\Pi\mb w-\mb w,\mb r)_{\mc T_h},\\
\label{eq:err_eqns_2}
(\nabla\times\bs\varepsilon_h^w,\mb v)_{\mc T_h}+\dualpr{\tau_t\mr P_N(\bs\varepsilon_h^u-\widehat{\bs\varepsilon}_h^u),\mb v}_{\pp\mc T_h}&\\
\nonumber
-(\varepsilon_h^p,\nabla\cdot\mb v)_{\mc T_h}+\dualpr{\widehat{\varepsilon}_h^p,\mb v\cdot\mb n}_{\pp\mc T_h}
&=\dualpr{\bs\delta_{\tau_t}^\Pi,\mb v}_{\pp\mc T_h},\\
\label{eq:err_eqns_3}
(\nabla\cdot\bs\varepsilon_h^u,q)_{\mc T_h}+\dualpr{\tau_n(\varepsilon_h^p-\widehat{\varepsilon}_h^p),q}_{\pp\mc T_h}
&=\dualpr{\delta_{\tau_n}^\Pi,q}_{\pp\mc T_h},\\
\label{eq:err_eqns_4}
-\dualpr{\mb n\times\bs\varepsilon_h^w+\tau_t(\bs\varepsilon_h^u-\widehat{\bs\varepsilon}_h^u),\bs\eta}_{\pp\mc T_h\backslash\Gamma}&=-\dualpr{\bs\delta_{\tau_t}^\Pi,\bs\eta}_{\pp\mc T_h\backslash\Gamma},\\
\label{eq:err_eqns_5}
-\dualpr{\widehat{\bs\varepsilon}_h^u,\bs\eta}_\Gamma &= 0,\\
\label{eq:err_eqns_6}
-\dualpr{\bs\varepsilon_h^u\cdot\mb n+\tau_n(\varepsilon_h^p-\widehat{\varepsilon}_h^p),\mu}_{\pp\mc T_h\backslash\Gamma}
&=-\dualpr{\delta_{\tau_n}^\Pi,\mu}_{\pp\mc T_h\backslash\Gamma},\\
\label{eq:err_eqns_7}
-\dualpr{\widehat{\varepsilon}_h^p,\mu}_\Gamma&=0,
\end{alignat}
\end{subequations}
for all $(\mb r,\mb v,q,\bs\eta,\mu)
\in W_h\times V_h\times Q_h\times N_h\times M_h$.
\end{lemma}
\begin{proof}
We obtain the error equations by taking the difference between equations \eqref{eq:proj_eq} and equations \eqref{eq:HDG_scheme}.
\end{proof}

\begin{lemma}
We have
\begin{subequations}\label{eq:dpr_eq}
\begin{alignat}{5}
\label{eq:dpr_eq_1}
(\Pi^*\mb w^*,\mb r)_{\mc T_h}+(\Pi^*\mb u^*,\nabla\times\mb r)_{\mc T_h}+\dualpr{\mr P_N\mb u^*,\mb r\times\mb n}_{\pp\mc T_h}&=(\Pi^*\mb w^*-\mb w^*,\mb r)_{\mc T_h},\\
\label{eq:dpr_eq_2}
-(\nabla\times\Pi^*\mb w^*,\mb v)_{\mc T_h}+\dualpr{\tau_t\mr P_N(\Pi^*\mb u^*-\mr P_N\mb u^*),\mb v}_{\pp\mc T_h}&\\
\nonumber
+(\Pi^* p^*,\nabla\cdot\mb v)_{\mc T_h}-\dualpr{\mr P_Mp^*,\mb v\cdot\mb n}_{\pp\mc T_h}
&=(\bs\theta,\mb v)_{\mc T_h}-\dualpr{\bs\delta_{-\tau_t}^{\Pi^*},\mb v}_{\pp\mc T_h},\\
\label{eq:dpr_eq_3}
-(\nabla\cdot\Pi^*\mb u^*,q)_{\mc T_h}+\dualpr{\tau_n(\Pi^* p^*-\mr P_Mp^*),q}_{\pp\mc T_h}
&=-\dualpr{\delta_{-\tau_n}^{\Pi^*},q}_{\pp\mc T_h},\\
\label{eq:dpr_eq_4}
\dualpr{\mb n\times\Pi^*\mb w^*-\tau_t(\Pi^*\mb u^*-\mr P_N\mb u^*),\bs\eta}_{\pp\mc T_h\backslash\Gamma}&=\dualpr{\bs\delta_{-\tau_t}^{\Pi^*},\bs\eta}_{\pp\mc T_h\backslash\Gamma},\\
\label{eq:dpr_eq_5}
\dualpr{\mr P_N\mb u^*,\bs\eta}_\Gamma &= 0,\\
\label{eq:dpr_eq_6}
\dualpr{\Pi\mb u^*\cdot\mb n-\tau_n(\Pi^* p^*-\mr P_Mp^*),\mu}_{\pp\mc T_h\backslash\Gamma}
&=\dualpr{\delta_{-\tau_n}^{\Pi^*},\mu}_{\pp\mc T_h\backslash\Gamma},\\
\label{eq:dpr_eq_7}
\dualpr{\mr P_Mp,\mu}_\Gamma&=0,
\end{alignat}
\end{subequations}
for all $(\mb r,\mb v,q,\bs\eta,\mu)
\in W_h\times V_h\times Q_h\times N_h\times M_h$.
\end{lemma}
\begin{proof}
The proof here is similar to the proof of \eqref{eq:proj_eq}.
\end{proof}

\begin{proof}[Proof of Proposition \ref{prop:ene_id}]
By \eqref{eq:err_eqns_4}, \eqref{eq:err_eqns_5}, \eqref{eq:err_eqns_6} and \eqref{eq:err_eqns_7} we have
\begin{align*}
-\dualpr{\mb n\times\bs\varepsilon_h^w+\tau_t(\bs\varepsilon_h^u-\widehat{\bs\varepsilon}_h^u),\widehat{\bs\varepsilon}_h^u}_{\pp\mc T_h}&=-\dualpr{\bs\delta_{\tau_t}^\Pi,\widehat{\bs\varepsilon}_h^u}_{\pp\mc T_h},\\
-\dualpr{\Pi\mb u\cdot\mb n+\tau_n(\varepsilon_h^p-\widehat{\varepsilon}_h^p),\widehat{\varepsilon}_h^p}_{\pp\mc T_h}
&=-\dualpr{\delta_{\tau_n}^\Pi,\widehat{\varepsilon}_h^p}_{\pp\mc T_h}.
\end{align*}
Now adding the above two equations with equations \eqref{eq:err_eqns_1} - \eqref{eq:err_eqns_3} with test functions $\mb r=\bs\varepsilon_h^w$, $\mb v=\bs\varepsilon_h^u$ and $q=\varepsilon_h^p$, we obtain the energy identity.

\end{proof}

\begin{proof}[Proof of Proposition \ref{prop:duality_id}]
Taking $\mb r=\Pi^*\mb w^*$, $\mb v=\Pi^*\mb u^*$, $q=\Pi^*p^*$ in equations \eqref{eq:err_eqns_1}-\eqref{eq:err_eqns_3}, and then combining \eqref{eq:err_eqns_4}, \eqref{eq:dpr_eq_5}, \eqref{eq:err_eqns_6} and \eqref{eq:dpr_eq_7}, we have
\begin{subequations}\label{eq:err_tested}
\begin{align}
(\bs\varepsilon_h^w,\Pi^*\mb w^*)_{\mc T_h}-(\bs\varepsilon_h^u,\nabla\times\Pi^*\mb w^*)_{\mc T_h}-\dualpr{\widehat{\bs\varepsilon}_h^u,\Pi^*\mb w^*\times\mb n}_{\pp\mc T_h}&=(\Pi\mb w-\mb w,\Pi^*\mb w^*)_{\mc T_h},\\
(\nabla\times\bs\varepsilon_h^w,\Pi^*\mb u^*)_{\mc T_h}+\dualpr{\tau_t\mr P_N(\bs\varepsilon_h^u-\widehat{\bs\varepsilon}_h^u),\Pi^*\mb u^*}_{\pp\mc T_h}&\\
\nonumber
-(\varepsilon_h^p,\nabla\cdot\Pi^*\mb u^*)_{\mc T_h}+\dualpr{\widehat{\varepsilon}_h^p,\Pi^*\mb u^*\cdot\mb n}_{\pp\mc T_h}
&=\dualpr{\bs\delta_{\tau_t}^\Pi,\Pi^*\mb u^*}_{\pp\mc T_h},\\
(\nabla\cdot\bs\varepsilon_h^u,\Pi^*p^*)_{\mc T_h}+\dualpr{\tau_n(\varepsilon_h^p-\widehat{\varepsilon}_h^p),\Pi^*p^*}_{\pp\mc T_h}
&=\dualpr{\delta_{\tau_n}^\Pi,\Pi^*p^*}_{\pp\mc T_h},\\
-\dualpr{\mb n\times\bs\varepsilon_h^w+\tau_t(\bs\varepsilon_h^u-\widehat{\bs\varepsilon}_h^u),\mr P_N\mb u^*}_{\pp\mc T_h}&=-\dualpr{\bs\delta_{\tau_t}^\Pi,\mr P_N\mb u^*}_{\pp\mc T_h},\\
-\dualpr{\bs\varepsilon_h^u\cdot\mb n+\tau_n(\varepsilon_h^p-\widehat{\varepsilon}_h^p),\mr P_Mp}_{\pp\mc T_h}
&=-\dualpr{\delta_{\tau_n}^\Pi,\mr P_Mp}_{\pp\mc T_h}.
\end{align}
\end{subequations}

On the other hand, taking $\mb r=\bs\varepsilon_h^w$, $\mb v=\bs\varepsilon_h^u$, $q=\varepsilon_h^p$ in equations \eqref{eq:dpr_eq_1}-\eqref{eq:dpr_eq_3}, and then combining \eqref{eq:dpr_eq_4}, \eqref{eq:err_eqns_5}, \eqref{eq:dpr_eq_6} and \eqref{eq:err_eqns_7}, we have
\begin{subequations}\label{eq:dl_tested}
\begin{alignat}{5}
(\Pi^*\mb w^*,\bs\varepsilon_h^w)_{\mc T_h}+(\Pi^*\mb u^*,\nabla\times\bs\varepsilon_h^w)_{\mc T_h}+\dualpr{\mr P_N\mb u^*,\bs\varepsilon_h^w\times\mb n}_{\pp\mc T_h}&=(\Pi^*\mb w^*-\mb w^*,\bs\varepsilon_h^w)_{\mc T_h},\\
-(\nabla\times\Pi^*\mb w^*,\bs\varepsilon_h^u)_{\mc T_h}+\dualpr{\tau_t\mr P_N(\Pi^*\mb u^*-\mr P_N\mb u^*),\bs\varepsilon_h^u}_{\pp\mc T_h}&\\
\nonumber
+(\Pi^* p^*,\nabla\cdot\bs\varepsilon_h^u)_{\mc T_h}-\dualpr{\mr P_Mp^*,\bs\varepsilon_h^u\cdot\mb n}_{\pp\mc T_h}
&=(\bs\theta,\bs\varepsilon_h^u)_{\mc T_h}-\dualpr{\bs\delta_{-\tau_t}^{\Pi^*},\bs\varepsilon_h^u}_{\pp\mc T_h},\\
-(\nabla\cdot\Pi^*\mb u^*,\varepsilon_h^p)_{\mc T_h}+\dualpr{\tau_n(\Pi^* p^*-\mr P_Mp^*),\varepsilon_h^p}_{\pp\mc T_h}
&=-\dualpr{\delta_{-\tau_n}^{\Pi^*},\varepsilon_h^p}_{\pp\mc T_h},\\
\dualpr{\mb n\times\Pi^*\mb w^*-\tau_t(\Pi^*\mb u^*-\mr P_N\mb u^*),\widehat{\bs\varepsilon}_h^u}_{\pp\mc T_h}&=\dualpr{\bs\delta_{-\tau_t}^{\Pi^*},\widehat{\bs\varepsilon}_h^u}_{\pp\mc T_h},\\
\dualpr{\Pi\mb u^*\cdot\mb n-\tau_n(\Pi^* p^*-\mr P_Mp^*),\widehat{\varepsilon}_h^p}_{\pp\mc T_h}
&=\dualpr{\delta_{-\tau_n}^{\Pi^*},\widehat{\varepsilon}_h^p}_{\pp\mc T_h}.
\end{alignat}
\end{subequations}
Note that the left of equations \eqref{eq:err_tested} is the permutation of the left of equations \eqref{eq:dl_tested}. This completes the proof.

\end{proof}

\subsection{Proofs in Section \ref{sec:gal_proj}}

\begin{proof}[Proof of Theorem \ref{thm:curl+_conv}]
We first prove $\Pi_k^c$ is well defined. Note that $\dim\nabla\times\mc P_k(K)^3=\dim\mc P_k(K)^3-\dim\nabla\mc P_{k+1}(K)$. Denote by $d_k=\dim\mc P_k(K)$ and we have $\dim\nabla\times\mc P_k(K)^3 = 3d_k-(d_{k+1}-1)$. Similarly we obtain $\dim(\nabla\times\mc P_{k+1}(K)^3)^{\perp_k}=3d_k-3d_{k+1}+d_{k+2}-1$. Finally note that $\dim(\mc P_k(K)^3\oplus\nabla\widetilde{\mc P}_{k+2}(K))^{\perp_{k+1}}=3d_{k+1}-3d_k-(d_{k+2}-d_{k+1})$. Adding up the dimensions we know that the number of equations is equal to $3d_k$ and therefore \eqref{eq:curl+_proj} is a square system. Assume $\mb w=0$, it remains to show  the following system about $\mb w_K\in\mc P_k(K)^3$ only admits trivial solution:
\begin{subequations}\label{eq:curl+_uniq}
\begin{alignat}{5}\label{eq:curl+_uniq_1}
(\mb w_K,\mb r)_K&=0 \quad\forall\mb r\in\nabla\times\mc P_k(K)^3,\\
\label{eq:curl+_uniq_2}
(\mb w_K,\mb r)_K&=0 \quad\forall\mb r\in(\nabla\times\mc P_{k+1}(K)^3)^{\perp_k},\\
\label{eq:curl+_uniq_3}
(\mb w_K,\nabla\times\mb v)_{K}
&=0
\quad\forall\mb v\in(\mc P_k(K)^3\oplus\nabla\widetilde{\mc P}_{k+2}(K))^{\perp_{k+1}}.
\end{alignat}
\end{subequations}
Note that \eqref{eq:curl+_uniq_1} implies $(\mb w_K,\nabla\times\mb v)_K=0$ for all $\mb v\in\mc P_k(K)^3\oplus\nabla\widetilde{\mc P}_{k+2}(K)$. This with \eqref{eq:curl+_uniq_3} gives $(\mb w_K,\nabla\times\mc P_{k+1}(K)^3)_K=0$, which with \eqref{eq:curl+_uniq_2} implies $\mb w_K=0$. Therefore $\Pi_k^c$ is well defined.

We next prove \eqref{eq:curl+_conv}. Define $\bs\varepsilon_k^w:=\Pi_k^c\mb w-\Pi_k\mb w$. By
\eqref{eq:curl+_proj} we obtain
\begin{subequations}\label{eq:curl+_prj_conv}
\begin{alignat}{5}
\label{eq:curl+_prj_conv_1}
(\bs\varepsilon_k^w,\mb r)_K&=0 \quad\forall\mb r\in\nabla\times\mc P_k(K)^3\oplus(\nabla\times\mc P_{k+1}(K)^3)^{\perp_k},\\
\label{eq:curl+_prj_conv_2}
(\bs\varepsilon_k^w,\nabla\times\mb v)_{K}
&=\dualpr{(\mb n\times\mb w)-\mr P_N(\mb n\times\mb w),\mb v}_{\pp K}
\quad\forall\mb v\in(\mc P_k(K)^3\oplus\nabla\widetilde{\mc P}_{k+2}(K))^{\perp_{k+1}}.
\end{alignat}
\end{subequations}
From \eqref{eq:curl+_prj_conv_1} we have $(\bs\varepsilon_k^w,\nabla\times\mb v)_K=0$ for all $\mb v\in\mc P_k(K)^3\oplus\nabla\widetilde{\mc P}_{k+2}(K)$. Also note that
$(\mc P_k(K)^3\oplus\nabla\widetilde{\mc P}_{k+2}(K))^t\subset N(\pp K)$. Therefore we have
\begin{align}\label{eq:curl+_prj_conv_2_ext}
(\bs\varepsilon_k^w,\nabla\times\mb v)_{K}
=\dualpr{(\mb n\times\mb w)-\mr P_N(\mb n\times\mb w),\mb v}_{\pp K}
\quad\forall\mb v\in\mc P_{k+1}(K)^3.
\end{align}

Since $\bs\varepsilon_k^w\in\mc P_k(K)^3$, we can decompose $\bs\varepsilon_k^w=\bs\varepsilon_k^1+\bs\varepsilon_k^2$, where $\bs\varepsilon_k^1\in\nabla\times\mc P_{k+1}(K)^3$ and $\bs\varepsilon_k^2\in (\nabla\times\mc P_{k+1}(K)^3)^{\perp_k}$. Let $\mb v_1\in\mc P_{k+1}(K)^3$ such that $\bs\varepsilon_k^1=\nabla\times\mb v_1$. Choose any $p\in\mc P_{k+2}(K)$. Substituting $\mb v=\mb v_1+\nabla p$ in equation \eqref{eq:curl+_prj_conv_2_ext} and using \eqref{eq:curl+_prj_conv_1}, we have 
\begin{align*}
\|\bs\varepsilon_k^w\|_K^2
=(\bs\varepsilon_k^w,\nabla\times(\mb v_1+\nabla p)+\bs\varepsilon_k^2)_{K}
=\dualpr{(\mb n\times\mb w)-\mr P_N(\mb n\times\mb w),\mb v_1+\nabla p}_{\pp K}.
\end{align*}
Therefore
\begin{align*}
\|\bs\varepsilon_k^w\|_K^2
&\lesssim h_K^{-1/2}\|(\mb n\times\mb w)-\mr P_N(\mb n\times\mb w)\|_{\pp K}\inf_{p\in\mc P_{k+2}(K)}\|\mb v_1+\nabla p\|_K\\
&\lesssim h_K^{1/2}\|(\mb n\times\mb w)-\mr P_N(\mb n\times\mb w)\|_{\pp K}\|\nabla\times\mb v_1\|_K\\
&\le h_K^{1/2}\|(\mb n\times\mb w)-\mr P_N(\mb n\times\mb w)\|_{\pp K}\|\bs\varepsilon_k^w\|_K.
\end{align*}
Finally note that
\begin{align*}
\|\mb n\times\mb w-\mr P_N(\mb n\times\mb w)\|_{\pp K}\le 2\|\mb n\times\mb w-\mb n\times\Pi_k\mb w\|_{\pp K}\lesssim h_K^{m-1/2}|\mb w|_{m,K},
\end{align*}
with $m\in(\frac{1}{2},k+1]$, where we use \eqref{eq:conv_L2} for the last inequality sign. This completes the proof.

\end{proof}

\begin{proof}[Proof of Proposition \ref{prop:bdmh_proj}]
Let $\Pi_k^\mr{BDM}:H^1(K)^3\rightarrow\mc P_k(K)^3$ be the classical BDM projection (see \cite{Ne:1986}) and define $\bs\varepsilon_k^u:=\Pi_{k,\tau_K}^B\mb u-\Pi_k^\mr{BDM}\mb u$. Then we have
\begin{subequations}
\begin{align}
\label{eq:pj_bdmpr_1}
(\bs\varepsilon_k^u,\mb r)_K &=0\quad\forall\mb r\in \mc N_{k-2}(K),\\
\label{eq:pj_bdmpr_2}
\dualpr{\bs\varepsilon_k^u\cdot\mb n+\tau_K(\Pi_{k-1}p-p),\mu}_{\pp K}&=0\quad\forall\mu\in\mc R_k(\pp K).
\end{align}
\end{subequations}
Choosing $\mu=\bs\varepsilon_k^u\cdot\mb n$ in \eqref{eq:pj_bdmpr_2}, then we have
$\|\bs\varepsilon_k^u\cdot\mb n\|_{\pp K}\le \tau_K^\mr{max}\|\Pi_{k-1}p-p\|_{\pp K}$.
By \eqref{eq:pj_bdmpr_1} we know $\bs\varepsilon_h^u$ is the BDM lifting of $\bs\varepsilon_k^u\cdot\mb n$ and therefore $\|\bs\varepsilon_h^u\|_K\lesssim h_K^{1/2}\|\bs\varepsilon_k^u\cdot\mb n\|_{\pp K}\le \tau_K^\mr{max}h_K^m|p|_{m,K}$ with $m\in[1,k]$. Finally we use the well known convergence properties about the classical BDM projection, namely $\|\Pi_k^\mr{BDM}\mb u-\mb u\|_K\lesssim h_K^s|\mb u|_{s,K}$ with $s\in[1,k+1]$ and the proof is thus completed.
\end{proof}

\section*{Acknowledgments}
This work was partially supported by the NSF grant DMS-1818867. Shukai Du would like to thank P. Monk for helpful discussions that led to a better presentation of the paper.

\bibliographystyle{abbrv}
\bibliography{maxref_v2}

\end{document}